\documentclass{llncs}
\usepackage{graphicx,amssymb,amsmath}
\usepackage{comment}
\usepackage{graphv1}
\newtheorem{observation}{Observation}
\newtheorem{conj}{Conjecture}
\newcounter{claim}
\renewenvironment{claim}{\refstepcounter{claim}\medskip\par\noindent{\itshape Claim~\theclaim}.}{\medskip\\}
\newcommand{\Once}{\mathrm{Once}}
\newcommand{\Twice}{\mathrm{Twice}}
\newcommand{\Missing}{\mathrm{Missing}}
\newcommand{\Colors}{\mathrm{Colors}}

\newcommand{\com}[1]{}
\begin{document}
\title{The linear arboricity conjecture for graphs of low degeneracy}
\author{Manu Basavaraju\inst{1}\and Arijit Bishnu\inst{2} \and Mathew Francis\inst{3} \and Drimit Pattanayak\inst{2}}
\institute{National Institute of Technology Karnataka, Surathkal. Email: \texttt{manub@nitk.edu.in} \and Indian Statistical Institute, Kolkata.\\Email: \texttt{arijit@isical.ac.in,drimitpattanayak@gmail.com} \and Indian Statistical Institute, Chennai Centre. Email: \texttt{mathew@isichennai.res.in}}
\date{}
\bibliographystyle{plain}
\maketitle
\pagestyle{plain}
\begin{abstract}
A linear forest is an acyclic graph whose each connected component is a path; or in other words, it is an acyclic graph whose maximum degree is at most 2. A \emph{linear coloring} of a graph $G$ is an edge coloring of $G$ such that the edges in each color class form a linear forest.
The \emph{linear arboricity} of $G$, denoted as $\chi'_l(G)$, is the minimum number of colors required in any linear coloring of $G$. It is easy to see that for any graph $G$, $\chi'_l(G)\geq\left\lceil\frac{\Delta(G)}{2}\right\rceil$, where $\Delta(G)$ is the maximum degree of $G$. The Linear Arboricity Conjecture of Akiyama, Exoo and Harary from 1980 states that for every graph $G$, $\chi'_l(G)\leq \left \lceil \frac{\Delta(G)+1}{2}\right\rceil$. Basavaraju et al.~\cite{conf}
showed that the conjecture is true for 3-degenerate graphs and provided a linear time algorithm for computing a linear coloring using at most $\left\lceil\frac{\Delta(G)+1}{2}\right\rceil$ colors for any input 3-degenerate graph $G$. Recently, Chen, Hao and Yu~\cite{ChenHaoYu}
showed that $\chi'_l(G)=\left\lceil\frac{\Delta(G)}{2}\right\rceil$ for any $k$-degenerate graph $G$ having $\Delta(G)\geq 2k^2-k$.
From this result, we have $\chi'_l(G)=\left\lceil\frac{\Delta(G)}{2}\right\rceil$ 
for every 3-degenerate graph $G$ having $\Delta(G)\geq 15$. We show that this equality holds for every 3-degenerate graph $G$ having $\Delta(G)\geq 9$. Moreover, by extending the techniques used, we show a different proof for the Linear Arboricity Conjecture on 3-degenerate graphs.
This new proof has the advantage that it gives rise to a linear time algorithm that computes a linear coloring using at most $\left\lceil\frac{\Delta(G)+1}{2}\right\rceil$ colors for any input 3-degenerate graph $G$ which is simpler to implement and analyze than the one of Basavaraju et al.~\cite{conf}. Moreover, the linear coloring computed by the algorithm is optimal if $\Delta(G)\geq 9$, and the coloring has the property that the two edges incident on a vertex of degree 2 will always have different colors. Next, we prove that for every 2-degenerate graph $G$, $\chi'_l(G)=\left\lceil\frac{\Delta(G)}{2}\right\rceil$ if $\Delta(G)\geq 5$. We conjecture that this equality holds also when $\Delta(G)\in\{3,4\}$ and show that this is the case for some well-known subclasses of 2-degenerate graphs. These proofs can also be converted into linear time algorithms that generate optimal linear colorings for input 2-degenerate graphs of the respective types.
\end{abstract}

\section{Introduction}
All graphs considered in this paper are finite, simple and undirected. We denote the vertex set of a graph $G$ by $V(G)$ and its edge set by $E(G)$. For a vertex $u\in V(G)$, we denote by $N_G(u)$ the set of neighbors of $u$, i.e. $N_G(u)=\{v\in V(G)\colon uv\in E(G)\}$. We abbreviate $N_G(u)$ to just $N(u)$ when the graph $G$ is clear from the context. The \emph{degree} of a vertex $u\in V(G)$ is $|N(u)|$ and is denoted by $d_G(u)$, which is sometimes abbreviated to $d(u)$ when the graph $G$ is clear from the context. We denote by $\Delta(G)$ the \emph{maximum degree} of a graph $G$; i.e. $\Delta(G)=\max\{d_G(u)\colon u\in V(G)\}$. When the graph $G$ under consideration is clear, we sometimes abbreviate $\Delta(G)$ to just $\Delta$. We say that an edge $uv$ of a graph is ``incident to'' the vertices $u$ and $v$. For any terms not defined here, please refer to~\cite{diestel}.

An \emph{edge coloring} of a graph $G$ using the colors $\{1,2,\ldots,k\}$ is a mapping $c:E(G)\rightarrow\{1,2,\ldots,k\}$. Given an edge coloring $c$ of a graph $G$, we say that $c(e)$ ``is the color of'' the edge $e\in E(G)$, or that the edge $e$ ``has'' or ``is colored with'' the color $c(e)$. Given an edge coloring using colors $\{1,2,\ldots,k\}$, the \emph{color class $i$}, for some $i\in\{1,2,\ldots,k\}$, is the set of edges colored with the color $i$; i.e. it is the set $c^{-1}(i)=\{e\in E(G)\colon c(e)=i\}$. A \emph{proper edge coloring} is an edge coloring in which no two edges incident on a vertex are colored with the same color.

An acyclic graph, which is a graph whose each connected component is a tree, is also known as a \emph{forest}. The \emph{arboricity} of a graph is the minimum integer $k$ such that its edge set can be partitioned into $k$ forests. In other words, it is the minimum integer $k$ such that there is an edge coloring of the graph in which each color class is a forest. A \emph{linear forest} is a forest whose each connected component is a path; it is easily that linear forests are exactly the acyclic graphs having maximum degree at most 2. The \emph{linear arboricity} of a graph is the minimum number of linear forests into which its edge set can be partitioned.

A \emph{$k$-linear coloring} of a graph $G$ is an edge coloring of $G$ using the colors $\{1,2,\ldots,k\}$ in which each color class is a linear forest. Or in other words, it is an edge coloring in which every vertex has at most two edges of the same color incident to it and there is no ``monochromatic cycle'' in the graph --- a cycle whose edges all receive the same color. The linear arboricity of a graph $G$ is clearly the smallest integer $k$ such that it has a $k$-linear coloring and is denoted by $\chi'_l(G)$. The parameter $\chi'_l(G)$ was introduced by Harary~\cite{Harary}. The Linear Arboricity Conjecture, first stated by Akiyama, Exoo and Harary~\cite{A1}, is as follows.

\begin{conj}[Linear Arboricity Conjecture]
For every graph $G$, $$\chi'_l(G)\leq\left\lceil\frac{\Delta(G)+1}{2}\right\rceil.$$
\end{conj}

Note that for any graph $G$, $\chi'_l(G)\geq\left\lceil\frac{\Delta(G)}{2}\right\rceil$, since in any linear coloring of $G$, there can be at most 2 edges of the same color incident to any vertex. But every graph does not have a linear coloring using just these many colors; Harary~\cite{Harary} observes that if $G$ is a $\Delta(G)$-regular graph, then $\chi'_l(G)\geq\left\lceil\frac{\Delta(G)+1}{2}\right\rceil$. The Linear Arboricity Conjecture suggests that this lower bound for regular graphs is tight --- in fact, it is enough to prove the conjecture regular graphs, since every graph $G$ is the subgraph of a $\Delta(G)$-regular graph.
Please see~\cite{ChenHaoYu} or~\cite{conf} for a brief history of the conjecture and the partial results that have been obtained. The Linear Arboricity Conjecture is known to hold for $k$-degenerate graphs for $k\leq 3$ (see~\cite{conf}).
\medskip

For some graph classes, it is known that the linear arboricity matches the trivial lower bound of $\lceil\frac{\Delta}{2}\rceil$. Wu~\cite{Wu1} first showed that every planar graph $G$ having $\Delta(G)\geq 13$ satisfied this property; i.e. $\chi'_l(G)=\lceil\frac{\Delta(G)}{2}\rceil$. He also showed that every planar graph $G$ with girth $g(G)$ satisfied this property if $\Delta(G)\geq 7$ and $g(G)\geq 4$, $\Delta(G)\geq 5$ and $g(G)\geq 5$, or $\Delta(G)\geq 3$ and $g(G)\geq 6$. Later, Wu, Hou, and Liu~\cite{WuHouLiu} improved the first of these bounds by showing that every planar graph $G$ having $\Delta(G)\geq 7$ and does not contain cycles of length $i$, for some $i\in\{3,4,5\}$, has linear arboricity $\left\lceil\frac{\Delta(G)}{2}\right\rceil$. Finally, Chen, Tan, Wu, and Li~\cite{CTWL} showed that this property is also satisfied by planar graphs $G$ having $\Delta(G)\geq 7$ and no cycles of length 5 having chords, and also by planar graphs $G$ having $\Delta(G)\geq 5$ and no cycles of length 5 and 6 which have chords.
Cygan, Kowalik, and Lu\v{z}ar~\cite{Cygan} improved the first result of Wu by showing that the linear arboricity of all planar graphs $G$ which have $\Delta(G)\geq 9$ is $\left\lceil\frac{\Delta(G)}{2}\right\rceil$. 
They also showed an $O(n\log n)$ algorithm that produces a linear coloring of every planar graph on $n$ vertices with the optimum number of colors when $\Delta(G)\geq 9$. Further, they conjecture that $\chi'_l(G)=\left\lceil\frac{\Delta(G)}{2}\right \rceil$ for every planar graph $G$ having $\Delta(G)\geq 5$. They note that there are planar graphs having maximum degree less than 5 for which the equality given by the conjecture does not hold.

\subsubsection{Degenerate graphs.}
A graph $G$ is said to be \emph{$t$-degenerate} if every subgraph of $G$ contains a vertex of degree at most $t$.
Notice that by definition, a $t$-degenerate graph is also a $t'$-degenerate graph for every $t'\geq t$. We shall be mainly concerned with 2-degenerate and 3-degenerate graphs in this paper. Notice that the graphs in these classes are ``sparse'': It is not difficult to see that the maximum number of edges possible in a $t$-degenerate graph is $tn-{t+1\choose 2}$, for all $n\geq t$, and ${n\choose 2}$ for all $n\leq t$.
The classes of 2-degenerate graphs and 3-degenerate graphs are quite well studied in the literature and they have been shown to be related to many well known graph classes. For example, outerplanar graphs and series-parallel graphs (or the graphs having treewidth at most 2; also called partial 2-trees) are well known examples of subclasses of the class of 2-degenerate graphs. In fact, every graph of treewidth at most $t$ is $t$-degenerate, though there are even 2-degenerate graphs with arbitrarily large treewidth (for example, large grids). The class of 3-degenerate graphs contains interesting graph classes such as triangle-free planar graphs and graphs of treewidth at most 3 (also called partial 3-trees). 

Chen, Hao and Yu~\cite{ChenHaoYu} showed that $\chi'_l(G)=\left\lceil\frac{\Delta(G)}{2}\right \rceil$ for every $k$-degenerate graph $G$ having $\Delta(G)\geq 2k^2-k$, and raise the question of whether the lower bound on $\Delta(G)$ can be reduced to a linear function of $k$, or more specifically, just $2k$. Note that showing that $\chi'_l(G)=\left\lceil\frac{\Delta(G)}{2}\right \rceil$ for a $k$-degenerate graph $G$ becomes harder for smaller values of $\Delta(G)$ --- in fact, as noted before, $\chi'_l(G)=\left\lceil\frac{\Delta(G)+1}{2}\right\rceil$ when $\Delta(G)=k$, or in other words, when $G$ is a $k$-regular graph.

\subsubsection{Our results.}
We study the linear arboricity of 3-degenerate and 2-degenerate graphs with a specific interest in understanding when such graphs have a $\lceil\frac{\Delta}{2}\rceil$-linear coloring. From the result of Chen, Hao, and Yu, we get that every 3-degenerate graph $G$ having $\Delta(G)\geq 15$ and every 2-degenerate graph $G$ having $\Delta(G)\geq 5$ satisfy $\chi'_l(G)=\left\lceil\frac{\Delta(G)}{2}\right\rceil$. We first show that the linear arboricity of every 3-degenerate graph $G$ having $\Delta(G)\geq 9$ is $\left\lceil\frac{\Delta(G)}{2}\right\rceil$. The linear colorings that we construct for such graphs have the additional property that they contain no ``monochromatic vertices'' --- vertices of degree 2 that have two edges of the same color incident on them. Note that this result automatically implies that such linear colorings also exist for every planar graph $G$ having $\Delta(G)\geq 9$ and $g(G)\geq 4$, since such graphs are 3-degenerate. We then show that the arguments used can be adapted to give a new proof for the Linear Arboricity Conjecture for all 3-degenerate graphs. Again, the linear colorings obtained contain no monochromatic vertices except when we construct 2-linear colorings, in which case we get at most one monochromatic vertex per connected component of the graph.
We then give a new proof for the known fact (from~\cite{ChenHaoYu}) that every 2-degenerate graph $G$ has $\chi'_l(G)=\left\lceil\frac{\Delta(G)}{2}\right\rceil$ if $\Delta(G)\geq 5$.
Since the lower bound on $\Delta(G)$ cannot be reduced to 2 (a cycle does not have a 1-linear coloring), the interesting question is whether the lower bound on $\Delta(G)$ can be reduced to 3. We conjecture that $\chi'_l(G)=2$ for every 2-degenerate graph $G$ when $\Delta(G)\in\{3,4\}$. As evidence towards this conjecture, we prove the existence of 2-linear colorings in the following subclasses of 2-degenerate graphs $G$ having $\Delta(G)\leq 4$: (i) those having at least $2|V(G)|-5$ edges, (ii) those that are bipartite, (iii) those having treewidth at most 2. An implication of our result for (i) is that every 2-degenerate graph on $n$ vertices with maximum degree at most 4 and having $2n-3$ edges (the maximum possible number of edges in a 2-degenerate graph) contains a Hamiltonian path. Further, in each of these three cases, the 2-linear colorings that we construct have some special properties: for (i), we get 2-linear colorings having at most one monochromatic vertex; for (ii), we get 2-linear colorings having no monochromatic vertices; and for (iii), we show that given any collection of pairwise disjoint pairs of degree 2 vertices in the graph, there is a 2-linear coloring in which at most one vertex from each pair is a monochromatic vertex.

All our proofs can be converted into linear time algorithms that generate linear colorings of input graphs using a number of colors that matches the upper bounds obtained. These algorithms also have the advantage that except for the case of (i) above, they are simpler to implement and analyze when compared to the algorithm of~\cite{conf}, which required some notions such as ``pseudo-$k$-linear colorings'' just to prove that the algorithm runs in linear time.

We leave open the question as to what is the lowest integer $t$ so that every 3-degenerate graph having $\Delta(G)\geq t$ has a $\left\lceil\frac{\Delta(G)}{2}\right\rceil$-linear coloring.
\section{Notation and preliminaries}
Given a graph $G$ and a set $S\subseteq V(G)$, we denote by $G-S$ the graph obtained by removing the vertices in $S$ from $G$, i.e. $V(G-S)=V(G)\setminus S$ and $E(G-S)=E(G)\setminus\{uv\colon u\in S\}$. When $S\subseteq E(G)$, we abuse notation to let $G-S$ denote the graph obtained by removing the edges in $S$ from $G$; i.e. $V(G-S)=V(G)$ and $E(G-S)=E(G)\setminus S$. In both cases, if $S=\{s\}$, we sometimes denote $G-S$ by just $G-s$. Given a graph $G$ and vertices $u,v\in V(G)$ such that $uv\notin E(G)$ and $N_G(u)\cap N_G(v)=\emptyset$, we let $G/(u,v)$ denote the graph obtained by ``identifying'' the vertex $v$ with $u$. That is, $V(G/(u,v))=V(G)\setminus\{v\}$ and $E(G/(u,v))=E(G-v)\cup\{ux\colon x\in N_G(v)\}$. 

The notation and observations from here till the end of this section also appear in~\cite{conf}.
\medskip

Let $G$ be a $t$-degenerate graph. A \emph{pivot} in $G$ is a vertex that has at most $t$ neighbors of degree more than $t$. A \emph{pivot edge} in $G$ is an edge between a pivot and a vertex with degree at most $t$. 

\begin{observation}\label{obs:pivot}
Every $t$-degenerate graph $G$ either has $\Delta(G)\leq t$ or contains a pivot of degree at least $t+1$.
\end{observation}
\begin{proof}
Suppose that $\Delta(G)>t$. Then the graph $G'=G-\{u\colon d_G(u)\leq t\}$ contains at least one vertex. Since $G'$ is also $t$-degenerate (as every subgraph of a $t$-degenerate graph is also $t$-degenerate), there exists a vertex $v\in V(G')$ such that $d_{G'}(v)\leq t$. It can be seen that $\{u\in N_G(v)\colon d_G(u)>t\}=N_{G'}(v)$. As $|N_{G'}(v)|=d_{G'}(v)\leq t$, we have that $v$ is a pivot in $G$. Since $v\in V(G')$, we know that $d_G(v)>t$, and hence we are done.\hfill\qed
\end{proof}

\noindent\textbf{Pivots and pivot edges.} Notice that every non-empty $t$-degenerate graph $G$ has at least one pivot and at least one pivot edge. This can be seen as follows.
If $\Delta(G)\leq t$, then every vertex of $G$ is a pivot, and every edge of $G$ is a pivot edge. If $\Delta(G)>t$, then by Observation~\ref{obs:pivot}, there exists a pivot $v$ in $G$ such that $d_G(v)>t$. Since a pivot has at most $t$ neighbors of degree more than $t$, there exists $u\in N_G(v)$ having $d_G(u)\leq t$. Then $uv$ is a pivot edge in $G$.
\medskip

Given a $k$-linear coloring of $G/(u,v)$, the vertex $u$ can be ``split back'' into the vertices $u$ and $v$ so as to obtain the graph $G$ together with a $k$-linear coloring of it. The following observation states this fact.
\begin{observation}\label{obs:id}
Let $G$ be a graph and $u,v\in V(G)$ such that $uv\notin E(G)$ and $N_G(u)\cap N_G(v)=\emptyset$.  If $c$ is a $k$-linear coloring of $G/(u,v)$, then $$c'(e)=\left\{\begin{array}{ll}c(e)&\mbox{if }e\mbox{ is not incident to }v\\c(ux)&\mbox{if }e=vx\end{array}\right.$$ is a $k$-linear coloring of $G$.
\end{observation}
We shall use the above observation without referring to it explicitly.
\medskip

\noindent\textbf{The sets $\boldsymbol{\Colors(x)}$, $\boldsymbol{\Missing(x)}$, $\boldsymbol{\Twice(x)}$, and $\boldsymbol{\Once(x)}$.}
Let $c$ be a $k$-linear coloring of a graph $G$. For a vertex $x\in V(G)$, let $\Colors(x)$ denote the set of colors in $\{1,2,\ldots,k\}$ that are the color of some edge incident to $x$. Further define $\Missing(x)$ to be the colors in $\{1,2,\ldots,k\}$ that do not appear on any edge incident to $x$, $\Twice(x)$ to be the set of colors that appear on two edges incident to $x$, and $\Once(x)$ to be the set of colors that appear on exactly one edge incident to $x$. Note that for any vertex $x\in V(G)$, $|\Missing(x)|+|\Once(x)|+|\Twice(x)|=k$ and also that the degree of $x$ in $G$ is $|\Once(x)|+2|\Twice(x)|$.
\medskip

\noindent\textbf{Monochromatic vertices.}
Let $G$ be a linearly colored graph. We say that a vertex in $G$ is \emph{monochromatic} if it has degree 2 and both the edges incident to it have the same color.

\section{Linear arboricity of 3-degenerate graphs}
In this section, we first prove that 
every 3-degenerate graph $G$ has a $\left\lceil\frac{\Delta(G)}{2}\right\rceil$-linear coloring containing no monochromatic vertices if $\Delta(G)\geq 9$. We then extend the techniques used to show that every connected 3-degenerate graph $G$ has a $\left\lceil\frac{\Delta(G)+1}{2}\right\rceil$-linear coloring that contains no monochromatic vertices unless $G$ is an odd cycle (in which case it has a 2-linear coloring containing exactly one monochromatic vertex).
\medskip

\subsection{Some preliminary definitions}\label{sec:prelims}
Let $G$ be a 3-degenerate graph having $\Delta(G)\geq 3$.
If $\Delta(G)=3$, then every vertex is a pivot and every edge is a pivot edge; in this case, we denote by $v$ some vertex having $d_G(v)=3$.
On the other hand, if $\Delta(G)>3$, we denote by $v$ a pivot having $d_G(v)\geq 4$, which is guaranteed to exist by Observation~\ref{obs:pivot}.
Let $F$ be the set of pivot edges incident on $v$. It can be seen that $F\neq\emptyset$. Define $X=\{x\in V(G):xv \in F\}$. Clearly, for each $x\in X$, we have $d_G(x)\leq 3$. Let $I=\{xx'\in E(G)\colon x,x'\in X$ and $d_G(x)=d_G(x')=2\}$. Let $H=G-(F\cup I)$. Let $W=\{x \in X\colon d_H(x)=1\}$. It is not difficult to see that $W$ is an independent set in $H$ as well as in $G$. For each vertex $w\in W$, let $\overline{w}$ denote its unique neighbor in $H$. In $H$, we now choose a maximal set of pairs of vertices from $W$ such that the vertices in each pair can be identified with each other without introducing multiple edges.
For this purpose, we define an auxiliary graph $A$ with $V(A)=W$ and $E(A)=\{xy\colon\overline{x}\neq\overline{y}\}$. Let $M$ be a maximal matching of $A$. Let $W'$ denote the vertices of $V(A)=W$ that remain unmatched by $M$. Due to the maximality of $M$, for any two vertices $a,b\in W'$, we have $\overline{a}=\overline{b}$.
Thus if $W'\neq\emptyset$, there exists a vertex $\widetilde{w}\in V(H)\setminus (W\cup\{v\})$ such that $\widetilde{w}=\overline{w}$ for every $w\in W'$.
We can now identify the vertices in $W$ which are matched to each other by $M$ to construct a graph $H'$ from $H$.

It is easy to see that the graph $H'$ is a 3-degenerate graph having $\Delta(H')\leq\Delta(G)$. Note that any linear coloring $c_{H'}$ of $H'$ that contains no monochromatic vertices can be converted into a linear coloring $c_H$ of $H$ containing no monochromatic vertices and using the same number of colors by just splitting back the vertices that were identified during the construction of $H'$ from $H$. We say that $c_H$ is the linear coloring of $H$ ``corresponding to'' the linear coloring $c_{H'}$ of $H'$.



\medskip

For any set $S\subseteq W$, we say that ``$S$ satisfies property $\mathcal{P}$'' if $|S|\geq 3$ and for each $x\in S$, if there exists $x'\in W$ such that $xx'\in M$, then we also have $x'\in S$.

\begin{claim}\label{clm:prop}
Let $c_{H'}$ be a linear coloring of $H'$ containing no monochromatic vertices, and let $c_H$ be the linear coloring of $H$ corresponding to $c_{H'}$. Let $S\subseteq W$. If $S$ satisfies property $\mathcal{P}$, then there exist $x,y \in S$ such that the colors of the edges $x\overline{x}$, $y\overline{y}$ are different in $c_H$.
\end{claim}
\textit{Proof of claim.} Suppose $S$ satisfies property $\mathcal{P}$. Assume for the sake of contradiction that for each $x\in S$, $c_H(x\overline{x})=1$ (say). Suppose that there exist $u,w\in S$ such that $uw\in M$. Then the edges $u\overline{u}$ and $w\overline{w}$ got their colors in $c_H$ from the colors of two edges incident on a degree 2 vertex in $c_{H'}$. Then this vertex has two edges of color 1 incident on it in $c_{H'}$, which contradicts the fact that there are no monochromatic vertices in $c_{H'}$. As $S$ satisfies property $\mathcal{P}$, we can now conclude that $S\subseteq W'$. As $|S|\geq 3$, we have three distinct vertices $p,q,r\in S$ such that $c_H(p\overline{p})=c_H(q\overline{q})=c_H(r\overline{r})=1$.
Since $S\subseteq W'$, we have that $\overline{p}=\overline{q}=\overline{r}=\widetilde{w}$. 
Thus $c_H(p\widetilde{w})=c_H(q\widetilde{w})=c_H(r\widetilde{w})=1$. Then three edges incident on the vertex $\widetilde{w}$ have the same color in $c_H$, contradicting the fact that $c_H$ is a linear coloring of $H$. This proves the claim.
\bigskip

\subsection{Maximum degree at least 9}

\begin{theorem}\label{thm:3-degrev}
Every 3-degenerate graph $G$ having maximum degree $\Delta(G)\leq k$, where $k\geq 9$, has a $\left\lceil \frac{k}{2} \right\rceil$-linear coloring in which there are no monochromatic vertices.
\end{theorem}
\begin{proof}
We prove this by induction on $|E(G)|$. 
If $\Delta(G)\leq 2$, then $G$ is a disjoint union of cycles and paths; in this case, $G$ clearly has a 3-linear coloring in which there is no monochromatic vertex, and we are done. So we can assume that $\Delta(G)\geq 3$. 

Let the sets $F$, $I$, $X$, $W$, and the graphs $H$, $H'$ be as defined in Section~\ref{sec:prelims}.
Since $|E(H')|<E(G)|$, and $\Delta(H')\leq\Delta(G)\leq k$, we have by the induction hypothesis that there is a $\left\lceil\frac{k}{2}\right\rceil$-linear coloring $c_{H'}$ of $H'$ in which there are no monochromatic vertices. Let $c_H$ be the linear coloring of $H$ corresponding to $c_{H'}$.
We show how to construct a $\left\lceil\frac{k}{2}\right\rceil$-linear coloring of $G$ that does not contain any monochromatic vertices, starting from the coloring $c_H$. We first construct a $\left\lceil\frac{k}{2}\right\rceil$-linear coloring $c$ of $G-I$ that does not contain any monochromatic vertices by extending $c_H$. Once this is done, the linear coloring $c$ can be easily further extended to the required $\left\lceil\frac{k}{2}\right\rceil$-linear coloring of $G$ by assigning to each edge $xx'\in I$ a color that is different from both $c(vx)$ and $c(vx')$ (this can be done since $\left\lceil\frac{k}{2}\right\rceil\geq 5$). We describe below how the $\left\lceil\frac{k}{2}\right\rceil$-linear coloring $c$ of $G-I$ is constructed. Note that $W$ is exactly the set of vertices in $X$ that have degree 2 in $G-I$. In the following, we denote by $d(u)$ the degree of a vertex $u$ in the graph $G-I$.

Let $\mathcal{C}$ denote the set of edge colorings (that are not necessarily linear colorings) of $G-I$ using colors in $\{1,2,\ldots,\left\lceil\frac{k}{2}\right\rceil\}$ that can be obtained from $c_H$ by coloring the edges in $F$ using the colors from $\Missing(v)\cup\Once(v)$ such that every color in $\Missing(v)$ is given to at most two edges of $F$ and every color in $\Once(v)$ is given to at most one edge of $F$ (here, the sets $\Missing(v)$ and $\Once(v)$ are with respect to the coloring $c_H$). In other words, $\mathcal{C}$ is the set of edge colorings of $G-I$ that are extensions of $c_H$ in which at most two edges of the same color are incident on $v$. Notice that $\mathcal{C}\neq\emptyset$, since we can always generate a coloring that belongs to $\mathcal{C}$ by the following procedure: each color in $\Missing(v)$ is assigned to some two edges of $F$, and each color in $\Once(v)$ is assigned to some edge in $F$, so that every edge in $F$ gets colored. This can be done because $d(v)\leq\Delta(G)\leq k$.
We make the following observation.
\medskip

{\itshape
Since $c_H$ did not contain any monochromatic vertex, in every coloring in $\mathcal{C}$, no vertex has more than two edges of the same color incident on it. 
}
\medskip

Thus, for any coloring in $\mathcal{C}$, the subgraph formed by the edges having the same color is a disjoint union of cycles and paths. Note that given any coloring in $\mathcal{C}$, permuting the colors on the edges in $F$ always gives another coloring in $\mathcal{C}$.

Among the colorings in $\mathcal{C}$, let $c$ denote a coloring that contains the smallest possible number of monochromatic vertices, and subject to that, contains the smallest possible number of monochromatic cycles. We claim that $c$ is a linear coloring of $G-I$ that does not contain any monochromatic vertices. From here onward, the sets $\Missing(v)$ and $\Once(v)$ shall be with respect to the coloring $c$.
 
Suppose that there exists a monochromatic vertex $u$ in $c$. Since there are no monochromatic vertices in $c_H$ and $d(v)\geq 3$, we can conclude that $u\in W$. Suppose first that there exists a color $i\in \Missing(v)\cup (\Once(v)\setminus\{c(uv)\})$. Then we can change the color of $uv$ to $i$ so that $u$ is no longer a monochromatic vertex. It is easy to see that we have not introduced any new monochromatic vertex, and hence we have a coloring in $\mathcal{C}$ that has fewer monochromatic vertices than $c$, which contradicts our choice of $c$. So we can assume that $\Missing(v)=\emptyset$ and that $\Once(v)\subseteq\{c(uv)\}$, which implies that $|\Missing(v)\cup\Once(v)|\leq 1$. This means that $d(v)\geq k-1$ (note that this immediately implies that $d(v)\geq 8$). As $v$ is a pivot, it follows that $|F|=|X|\geq k-4$.

Next, suppose that there exists $x\in X\setminus W$. Then $d(x)\in\{1,3\}$. If $c(uv)\neq c(xv)$, then we can interchange the colors of $uv$ and $xv$ so as to obtain a coloring in $\mathcal{C}$ having a smaller number of monochromatic vertices than $c$, which contradicts our choice of $c$. So we can assume that for every $y\in X\setminus W$, $c(uv)=c(yv)$. Since at most one edge in $F$ other than $uv$ can have the color $c(uv)$, we can now assume that $X\setminus W=\{x\}$, and also that $c(xv)=c(uv)$. Note that this implies that for every $y\in W\setminus\{u\}$, $c(yv)\neq c(uv)$.
Since $|X|\geq k-4$, we now have $|W|\geq k-5$. Let $S=W$. As $k\geq 9$, we then have $|S|\geq 4$.
Clearly, $S$ satisfies property $\mathcal{P}$. By Claim~\ref{clm:prop}, we know that there exists $y\in S$ such that $c(y\overline{y})\neq c(u\overline{u})=c(uv)$ (note that $y\neq u$). We can then interchange the colors of $yv$ and $uv$ to obtain a coloring in $\mathcal{C}$ that has fewer monochromatic vertices than $c$, which contradicts the choice of $c$. So we can assume that $X\setminus W=\emptyset$, or in other words $X=W$. If there exists $w\in W\setminus\{u\}$ such that $c(vw)=c(uv)$, then let $S=W\setminus(\{w\}\cup\{w'\in W\colon ww'\in M\})$; otherwise, let $S=W$.
It is easy to see that $|S|\geq |W|-2=|X|-2\geq k-6\geq 3$. Thus $S$ satisfies property $\mathcal{P}$. By Claim~\ref{clm:prop}, it follows that there exists $y\in S$ such that $c(y\overline{y})\neq c(u\overline{u})=c(uv)$. Now, by interchanging the colors of the edges $uv$ and $yv$, we can obtain a coloring in $\mathcal{C}$ that has fewer monochromatic vertices than $c$, which contradicts the choice of $c$. We can thus conclude that there are no monochromatic vertices in $c$.

Next let us suppose that $c$ contains a monochromatic cycle. Let each edge of this cycle have color~1~(say). As there are no monochromatic cycles in $c_H$, we can infer that an edge $uv\in F$ is contained in this monochromatic cycle. Let $u'v$ be the other edge incident to $v$ that is contained in the monochromatic cycle. Clearly, $c(uv)=c(u'v)=1$. As there are no monochromatic vertices in $c$, we have $d(u)\geq 3$ and $d(u')\geq 3$, which implies that $u,u'\notin W$. As $uv\in F$, we further have that $d(u)=3$. Let $2$ be the color of the edge incident on $u$ that is not part of the monochromatic cycle. We denote by $P$ the monochromatic path from $v$ to $u$ containing only edges colored 1, and not containing the edge $uv$; i.e. it is the path obtained by removing the edge $uv$ from the monochromatic cycle.
Since there is only one edge colored 2 incident on $u$, we know that there is a maximal monochromatic path, all of whose edges are colored 2, having one endpoint $u$. Let us denote this path by $Q$.

Suppose that there exists $i\in\Missing(v)\cup(\Once(v)\setminus\{2\})$. Then we can change the color of $uv$ to $i$ so as to obtain a coloring in $\mathcal{C}$ that does not contain any monochromatic vertices and contains fewer monochromatic cycles than $c$. As this is a contradiction to the choice of $c$, we can assume that $\Missing(v)=\emptyset$ and $\Once(v)\subseteq\{2\}$.
This implies that $d(v)\geq k-|\Once(v)|$. As before, since $v$ is a pivot, we have $|F|=|X|\geq k-|\Once(v)|-3$. Note that since $|\Once(v)|\leq 1$, this means that $|F|=|X|\geq k-4$.

\begin{claim}\label{clm:exch}
Let $xv\in F$ such that $x\notin\{u,u'\}$. Let $c'$ be the coloring in $\mathcal{C}$ obtained from $c$ by exchanging the colors of the edges $uv$ and $xv$. If $c'$ does not have fewer monochromatic cycles than $c$, then $c(xv)=c'(uv)=2$ and $c'$ contains a monochromatic cycle colored 2 containing the edge $uv$.
\end{claim}
\textit{Proof of claim.}
It is clear that the monochromatic cycle colored 1 in $c$ is no longer a monochromatic cycle in $c'$, since $c'(uv)=c(xv)\neq 1$. If there is a new monochromatic cycle in $c'$, then clearly, it has to contain either the edge $xv$ or the edge $uv$. In the former case, i.e. there is a monochromatic cycle colored 1 in $c'$ containing the edge $xv$, since $P$ is a path colored 1 from $v$ to $u$ in $c'$ as well, we have that $u$ also belongs to this cycle. But this contradicts the fact that there is only one edge colored 1 incident on $u$ in $c'$. Thus, if at all a new monochromatic cycle arises in $c'$, it has to be one containing the edge $uv$. Since the only color other than 1 that appeared on the edges incident on $u$ in $c$ was 2, it follows that this new monochromatic cycle is colored 2, which implies that $c'(uv)=c(xv)=2$. This proves the claim.
\medskip

First, suppose that $v$ is not contained in $Q$. If there exists $x\in X\setminus (W\cup\{u,u'\})$, then we exchange the colors of the pivot edges $xv$ and $uv$ to obtain a new coloring $c'$ in $\mathcal{C}$. Clearly, the coloring $c'$ does not contain any monochromatic vertices. By Claim~\ref{clm:exch} and our choice of $c$, we have that $c'$ contains a monochromatic cycle colored 2 containing the edge $uv$, which we shall denote by $C$. Then $C-uv$ is a path in $c'$, all of whose edges are colored 2, from $u$ to $v$. The edge $xv$ is not on this path since $c'(xv)=1$, and therefore $C-uv$ is a path colored 2 in $c$ too, implying that $v$ is contained in the path $Q$ in $c$, contradicting our assumption that $v$ is not contained in $Q$. So we can assume that $X\setminus W\subseteq\{u,u'\}$. As $|X|\geq k-4$, we now have $|W|\geq k-6\geq 3$. Let $S=W$. It is easy to see that $S$ satisfies property $\mathcal{P}$, and therefore by Claim~\ref{clm:prop}, we have that there exists $y\in S$ such that $c(y\overline{y})\neq 1$. We now exchange the colors of the edges $uv$ and $yv$ to obtain a new coloring $c'$ in $\mathcal{C}$. By our choice of $y$, it follows that there are no monochromatic vertices in $c'$. Then by our choice of $c$, we have that $c'$ does not have fewer monochromatic cycles than $c$, which implies by Claim~\ref{clm:exch} that there is a monochromatic cycle colored 2 containing the edge $uv$ in $c'$. As before, this implies that there is a monochromatic path colored 2 between $u$ and $v$ in $c$, which contradicts our assumption that $v$ does not lie on $Q$.

So we can assume that $v$ is contained in $Q$.
Let $zv$ be the first edge on $Q$ that is incident on $v$ (when traversing the path $Q$ starting from $u$), and let $Q_z$ denote the subpath of $Q$ between $u$ and $z$. Since there are no monochromatic vertices in $c$, we know that $d(z)\geq 3$.
Suppose that $zv\in F$, i.e. $z\in X$. Then we exchange the colors of $zv$ and $uv$ to obtain a new coloring $c'$ in $\mathcal{C}$. It is easy to see that this does not create any monochromatic vertices. By our choice of $c$, we then have that $c'$ does not contain fewer monochromatic cycles than $c$. Then by Claim~\ref{clm:exch}, we know that $c'$ has a monochromatic cycle colored 2 containing the edge $uv$, which we shall denote by $C$. Note that even in the coloring $c'$, the path $Q_z$ is a monochromatic path colored 2 between $u$ and $z$. Thus, the fact that $u$ is contained in $C$ implies that $z$ is contained in $C$, which contradicts the fact that exactly one edge colored 2 is incident on $z$ in $c'$. We can thus conclude $zv\notin F$, or in other words, $z\notin X$.


If there exists $x\in X\setminus W$ such that $c(xv)\notin\{1,2\}$, then we can exchange the colors of the edges $xv$ and $uv$ to obtain a new coloring $c'$ in $\mathcal{C}$. Clearly, $c'$ does not contain any monochromatic vertices, and therefore by our choice of $c$, it should contain at least as many monochromatic cycles as $c$. Then by Claim~\ref{clm:exch}, we have that $c(xv)=2$, which contradicts the fact that $c(xv)\notin\{1,2\}$. We can thus assume that for every $x\in X\setminus W$, $c(xv)\in\{1,2\}$. Since there can be at most two edges of each color incident on $v$, and $zv\notin F$ is an edge colored 2 incident on $v$, we have that $|X\setminus W|\leq 3-|\Once(v)|$. Thus, $|W|\geq |X|-3+|\Once(v)|$. Recalling that $|X|\geq k-|\Once(v)|-3$, we now have that $|W|\geq k-6\geq 3$. Note that for every $y\in W$, we have $c(yv)\neq 1$, since $u,u'\notin W$. Suppose that there exists $y\in W$ such that $c(y\overline{y})\neq 1$ and $c(yv)\neq 2$, then we exchange the colors of $yv$ and $uv$ to obtain a new coloring $c'$ in $\mathcal{C}$. Clearly, $c'$ does not contain any monochromatic vertices, and therefore by the choice of $c$ and Claim~\ref{clm:exch}, it follows that $c'(uv)=2$, which contradicts the fact that $c(yv)\neq 2$. So we can assume that for every $y\in W$ such that $c(yv)\neq 2$, we have $c(y\overline{y})=1$.

Let $S=W$. Since $S$ satisfies property $\mathcal{P}$, we have by Claim~\ref{clm:prop} that there exists $y\in S$ such that $c(y\overline{y})\neq 1$. Then by the above observation, we have that $c(yv)=2$.
Since $zv$ and $yv$ are two edges colored 2 incident on $v$, we now have that for every $p\in X\setminus\{y\}$, $c(pv)\neq 2$, and also that $|\Once(v)|=0$. Then using our previous observation that for every $x\in X\setminus W$, $c(xv)\in\{1,2\}$, we can conclude that for every $x\in X\setminus W$, $c(xv)=1$. This implies that $|X\setminus W|\leq 2$. This gives $|W|\geq |X|-2\geq k-|\Once(v)|-5=k-5\geq 4$. Thus there exists $w\in W\setminus\{y\}$ such that $c(wv)\neq c(y\overline{y})$. Notice that $c(wv)\neq 2$, which implies by our observation from the previous paragraph that $c(w\overline{w})=1$. We now construct a new coloring $c'$ in $\mathcal{C}$ by setting $c'(uv)=c(wv)$, $c'(yv)=c(uv)=1$, $c'(wv)=c(yv)=2$, and by giving every other edge the same color as it has in $c$. Then $c'$ is a coloring in $\mathcal{C}$ containing no monochromatic vertices. Since $c(wv)\notin\{1,2\}$, we have that there is no monochromatic cycle containing $uv$ in $c'$. Since $y$ and $w$ are not monochromatic vertices in $c'$, it is clear that neither $yv$ nor $wv$ are contained in monochromatic cycles in $c'$. This implies that $c'$ contains fewer monochromatic cycles than $c$, which contradicts our choice of $c$.
\hfill\qed
\end{proof}
\begin{corollary}
Every 3-degenerate graph $G$ having $\Delta(G)\geq 9$ has a $\left\lceil\frac{\Delta(G)}{2}\right\rceil$-linear coloring.
\end{corollary}

\subsection{Maximum degree less than 9}

\begin{theorem}\label{thm:3dg7}
Every 3-degenerate graph $G$ having $\Delta(G)\leq 7$ has a 4-linear coloring containing no monochromatic vertices.
\end{theorem}
\begin{proof}
We prove this by induction on $|E(G)|$. Clearly, if $\Delta(G)\leq 2$, then $G$ has a 4-linear coloring containing no monochromatic vertices. So we assume that $\Delta(G)\geq 3$.

Let the sets $F$, $I$, $X$, $W$, and the graphs $H$, $H'$ be as defined in Section~\ref{sec:prelims}.
As $|E(H')|<|E(H)|$, by the induction hypothesis, there is a 4-linear coloring $c_{H'}$ of $H'$ that contains no monochromatic vertices. Let $c_H$ be the 4-linear coloring of $H$ containing no monochromatic vertices corresponding to $c_{H'}$. We assume that the colors used in $c_{H'}$ and $c_{H}$ are from $\{1,2,3,4\}$.
As in the earlier proof, let $\mathcal{C}$ denote the edge colorings using colors $\{1,2,3,4\}$ of $G-I$ obtained by extending $c_H$ by coloring the edges in $F$ in such a way that no color occurs more than twice on the edges incident on $v$.
Again, it is easy to see that any coloring in $\mathcal{C}$ has the property that each vertex has at most two edges of the same color incident on it. Let $c$ be a coloring in $\mathcal{C}$ containing the least number of monochromatic vertices, and subject to that, the least number of monochromatic cycles.
In the following, the sets $\Missing(v)$, $\Once(v)$, and $\Twice(v)$ are with respect to the coloring $c$, and we denote by $d(u)$ the degree of a vertex $u$ in $G-I$.

\begin{claim}\label{clm:nomv7}
There are no monochromatic vertices in $c$.
\end{claim} 
\textit{Proof of claim.}
For the sake of contradiction, let us assume that $u$ is a monochromatic vertex in $c$. Clearly, $u\in W$.
If there exists any color $i\in (\Missing(v)\cup\Once(v))\setminus\{c(uv)\}$, we can recolor $uv$ with $i$ to obtain a coloring in $\mathcal{C}$ having lesser number of monochromatic vertices than $c$, contradicting our choice of $c$. Therefore, we assume that $\Missing(v)=\emptyset$ and $\Once(v)\subseteq\{c(uv)\}$. This means that $|\Missing(v)\cup \Once(v)|\leq 1$, which implies that $d(v)\geq 7$. Since $\Delta(G)\leq 7$, we have $d(v)=7$ and therefore $\Once(v)=\{c(uv)\}$, and also that $|X|\geq 4$. If there is a vertex $w\in X \setminus W$, then we can exchange the colors of $wv$ and $uv$ to obtain a coloring in $\mathcal{C}$ that contains fewer monochromatic vertices than $c$, which again contradicts our choice of $c$. So we assume that $X=W$, which implies that $|W|\geq 4$. The set $W$ satisfies property $\mathcal{P}$. Then by Claim~\ref{clm:prop}, there exists $z\in W$ such that $c(z\bar{z})\neq c(u\bar{u})$. Now exchanging the colors on the edges $uv$ and $zv$ gives a coloring in $\mathcal{C}$ with fewer monochromatic vertices than $c$, again leading to the same contradiction (notice that $c(zv)\neq c(uv)$ since $\Once(v)=\{c(uv)\}$).
This proves the claim.

\begin{claim}\label{clm:nomc7}
There are no monochromatic cycles in $c$.
\end{claim}
\textit{Proof of claim.}
Let $C$ be a monochromatic cycle in $c$. Then as $c_H$ did not contain any monochromatic cycles, there exists $wv \in E(C)\cap F$. Using Claim~\ref{clm:nomv7}, we infer that $d(w)=3$, which means that $w\in X\setminus W$. Let us denote by 1 the color of the edges in $C$, and by 2 the color of the edge incident on $w$ that is not in $E(C)$. Let $P$ be the maximal path formed by edges of color 2 starting from $w$. Let $x$ denote the end vertex of $P$ other than $w$. Notice that since $d(v)\leq 7$, there exists $i\in\Once(v)\cup\Missing(v)$. Clearly, $i\neq 1$. If $x\neq v$, then we can recolor $wv$ with the color $i$ to obtain a coloring in $\mathcal{C}$ with no monochromatic vertices and fewer monochromatic cycles than $c$, which contradicts our choice of $c$. So we assume that $x=v$. If there is any $i\in (\Once(v)\cup\Missing(v))\setminus\{2\}$, then we can recolor $wv$ with $i$ to again obtain a coloring in $\mathcal{C}$ that contains no monochromatic vertices and fewer monochromatic cycles than $c$, which contradicts our choice of $c$. So we infer that $d(v)=7$ and $\Once(v)=\{2\}$. If there exists $u\in (X\setminus\{w\})\cap V(C)$, then recoloring $uv$ with 2 gives a coloring in $\mathcal{C}$ that contradicts the choice of $c$. Therefore, we can assume that $(X\setminus\{w\})\cap V(C)=\emptyset$. Then if there is any $u\in X\setminus\{w\}$, then $c(uv)\neq 1$, and exchanging the colors on the edges $wv$ and $uv$ gives a coloring in $\mathcal{C}$ that contradicts our choice of $c$. Hence we can assume that $X\setminus W=\{w\}$ and therefore $|W|\geq 3$. Then $W$ satisfies property $\mathcal{P}$, and by Claim~\ref{clm:prop}, there exists $u\in W$ such that $c(u\bar{u})\neq 1$. Now we can exchange the colors on the edges $uv$ and $wv$ to again get a coloring in $\mathcal{C}$ with no monochromatic vertices and fewer monochromatic cycles than $c$, which contradicts our choice of $c$. This proves the claim.
\medskip

From Claims~\ref{clm:nomv7} and~\ref{clm:nomc7}, it follows that $c$ is a 4-linear coloring of $G-I$ containing no monochromatic vertices. We can now extend $c$ to a 4-linear coloring of $G$ that does not contain any monochromatic vertices by coloring each edge $xx'\in I$ using a color that is different from $c(vx)$ and $c(vx')$.
This completes the proof.
\hfill\qed
\end{proof}

\begin{theorem}\label{thm:3dg5}
Every 3-degenerate graph $G$ having $\Delta(G)\leq 5$ has 3-linear coloring containing no monochromatic vertices.
\end{theorem}
\begin{proof}
As before, we prove this by induction on $|E(G)|$. Since $G$ has a 3-linear coloring containing no monochromatic vertices if $\Delta(G)\leq 2$, we assume that $\Delta(G)\geq 3$.

Let the sets $F$, $I$, $X$, $W$, and the graphs $H$, $H'$ be as defined in Section~\ref{sec:prelims}.
As $|E(H')|<|E(H)|$, by the induction hypothesis, there is a 3-linear coloring $c_{H'}$ of $H'$ that contains no monochromatic vertices. We now construct a new 3-linear coloring $c'_{H'}$ of $H'$, also containing no monochromatic vertices, by modifying $c_{H'}$ a little if required. Let $W=\{u,w\}$ and $\bar{u}=\bar{w}=x$ (say). Suppose that $c_{H'}(ux)=c_{H'}(wx)$. Since $x$ has degree at most 5 in $H'$, there is a color $i$ of $c_{H'}$ that does not occur twice on the edges incident on $x$. We now recolor one of the edges $ux$ or $wx$ in $c_{H'}$ with the color $i$ to obtain a new edge coloring $c'_{H'}$ of $H'$. It is easy to see that $c'_{H'}$ is also a 3-linear coloring of $H'$ having no monochromatic vertices. If $|W|\neq 2$ or if $W=\{u,w\}$, but $\bar{u}\neq\bar{w}$ or $c_{H'}(u\bar{u})\neq c_{H'}(w\bar{w})$, then we let $c'_{H'}=c_{H'}$. Thus, the coloring $c'_{H'}$ of $H'$ has the special property that if $W=\{u,w\}$ and $\bar{u}=\bar{w}$, then $c'_{H'}(u\bar{u})\neq c'_{H'}(w\bar{w})$.

Let $c_H$ be the 3-linear coloring of $H$ containing no monochromatic vertices corresponding to $c'_{H'}$. We assume that the colors used by $c'_{H'}$ and $c_H$ are from $\{1,2,3\}$.
We again let $\mathcal{C}$ denote the edge colorings using colors in $\{1,2,3\}$ of $G-I$ obtained by extending $c_H$ by coloring the edges in $F$ in such a way that no color occurs more than twice on the edges incident on $v$.
As before, every coloring in $\mathcal{C}$ has the property that each vertex has at most two edges of the same color incident on it. 
Let $c$ be a coloring in $\mathcal{C}$ containing the least number of monochromatic vertices, and subject to that, the least number of monochromatic cycles.
In the following, the sets $\Missing(v)$, $\Once(v)$, and $\Twice(v)$ are with respect to the coloring $c$, and we denote by $d(u)$ the degree of a vertex $u$ in $G-I$.

\begin{claim}\label{clm:nomv5}
There are no monochromatic vertices in $c$.
\end{claim}
\textit{Proof of claim.}
Let $u$ be a monochromatic vertex in $c$. Let $i\in (\Once(v)\cup\Missing(v))\setminus\{c(uv)\}$. Then we can recolor the edge $uv$ with the color $i$ to obtain a coloring in $\mathcal{C}$ that contains fewer monochromatic vertices than $c$, which contradicts our choice of $c$. So $\Missing(v)=\emptyset$ and $\Once(v)\subseteq\{c(uv)\}$. This implies that $d(v)\geq 5$. Since $\Delta(G)\leq 5$, we have $d(v)=5$, which implies that $\Once(v)=\{c(uv)\}$ and $|X|\geq 2$. If $w\in X\setminus W$, then we can exchange the colors of $uv$ and $wv$ to get a coloring in $\mathcal{C}$ with fewer monochromatic vertices, again contradicting our choice of $c$. So $|W|=|X|\geq 2$. If $c(uv)=c(u\bar{u})\neq c(w\bar{w})$ for some $w\in W$, then we can exchange the colors on the edges $uv$ and $wv$ to obtain a coloring in $\mathcal{C}$ which will lead to the usual contradiction to our choice of $c$. Hence it must be the case that $c(u\bar{u})=c(w\bar{w})$ for all $w\in W$. From Claim~\ref{clm:prop}, it follows that $|W|=2$. Let $W=\{u,w\}$. Since $c(u\bar{u})=c(w\bar{w})$, and the fact that there are no monochromatic vertices in $c'_{H'}$, we get that $\bar{u}=\bar{w}=x$ (say). Note that we have $c'_{H'}(u\bar{u})=c(u\bar{u})=c(w\bar{w})=c'_{H'}(w\bar{w})$. We now have a contradiction to the special property of $c'_{H'}$ that was observed above.
This proves the claim.
\medskip

\begin{claim}\label{clm:nomc5}
There are no monochromatic cycles in $c$.
\end{claim}
\textit{Proof of claim.}
Suppose that there is a monochromatic cycle $C$ in the coloring $c$ of $G-I$. Since there are no monochromatic cycles in $c_H$, we have that there exists $uv\in E(C)\cap F$. Let 1 denote the color of the edges of $C$. From Claim~\ref{clm:nomv5} it must be that $d(u)=3$. Let 2 denote the color of the edge incident on $u$ that does not belong to $C$. Let $P$ be the maximal path whose edges are colored 2 that starts at $u$, and let $x$ denote its end vertex other than $u$. Suppose first that $x\neq v$. Clearly, since $d(v)\leq 5$, there exists $i\in\Once(v)\cup\Missing(v)$. We can recolor the edge $uv$ with the color $i$ to obtain a coloring in $\mathcal{C}$ having no monochromatic vertices and having fewer monochromatic cycles than $c$, which contradicts our choice of $c$. So let us assume that $x=v$.
Now if there exists $i\in (\Once(v)\cup\Missing(v))\setminus\{2\}$, then we can recolor the edge $uv$ with $i$ to obtain a coloring in $\mathcal{C}$, which will again lead to the same contradiction to the choice of $c$. So we can assume that $\Missing(v)=\emptyset$ and $\Once(v)=\{2\}$. Note that this implies that $d(v)=5$.
Suppose there exists $w\in W$. Notice that since there are no monochromatic vertices in $c$, we have $c(wv)\notin\{1,2\}$. If $c(w\bar{w})\neq 1$, then we can exchange the colors on the edges $uv$ and $wv$ to obtain a coloring in $\mathcal{C}$ which will again lead to the same contradiction. On the other hand, if $c(w\bar{w})=1$, then we can color $uv$ with $c(wv)$ and $wv$ with 2 to get another coloring in $\mathcal{C}$ which will also lead to the same contradiction. So we can assume that $W=\emptyset$.
Since $d(v)=5$, this implies that $|X\setminus W|\geq 2$.
Thus there exists $w\in (X\setminus W)\setminus\{u\}$.
If $wv\in E(C)$, we can recolor the edge $wv$ with 2 to get a coloring in $\mathcal{C}$ which again will contradict our choice of $c$ as before. So we can assume that $wv\notin E(C)$, which means that $c(wv)\neq 1$. Then we can exchange the colors of the edges $uv$ and $wv$ to again get a coloring in $\mathcal{C}$ with no monochromatic vertices and fewer monochromatic cycles, contradicting our choice of $c$. This proves the claim.
\medskip

From Claims~\ref{clm:nomv5} and~\ref{clm:nomc5}, it follows that there is a 3-linear coloring of $G-I$ containing no monochromatic vertices. Now as before, we can color every edge $xx'\in I$ with a color not in $\{c(vx),c(vx')\}$ to obtain a 3-linear coloring of $G$ containing no monochromatic vertices.
\hfill\qed
\end{proof}

\begin{theorem}\label{thm:3dg3}
Every connected graph $G$ having $\Delta(G)\leq 3$ is either an odd cycle or has a 2-linear coloring without any monochromatic vertices.
\end{theorem}

\begin{proof}
We prove this by induction on $|E(G)|$.
Let $G$ be a connected graph of maximum degree 3.
Observe that if $G$ is an odd cycle, then for any vertex $u\in V(G)$, there is a 2-linear coloring of $G$ in which the only monochromatic vertex is $u$. If $G$ is an even cycle or a path, then any proper edge coloring of $G$ using 2 colors is a 2-linear coloring of $G$ having no monochromatic vertex.
Suppose that $G$ contains a vertex $u$ such that $d_G(u)=1$. Let $N_G(u)=\{v\}$. Let $H=G-u$. Clearly, $H$ is a connected graph. Notice that $|E(H)|<|E(G)|$. If $H$ is an odd cycle, then we have by the induction hypothesis that $H$ has a 2-linear coloring $c$ in which $v$ is the only monochromatic vertex. Otherwise, we have by the induction hypothesis that $H$ has a 2-linear coloring $c$ having no monochromatic vertices. In either case, we extend $c$ to a 2-linear coloring of $G$ containing no monochromatic vertices by coloring the edge $uv$ with a color in $\Once(v)\cup\Missing(v)$.

So we assume that $G$ contains no vertex of degree 1. If there is no vertex of degree more than 2, then $G$ is a cycle, in which case we are already done as noted above. So there exists $v\in V(G)$ such that $d_G(v)=3$. Let $u\in N_G(v)$. Let $H=G-uv$.

First, suppose that $H$ is disconnected. Clearly, $H$ has two connected components, say $C_u$ and $C_v$, containing $u$ and $v$ respectively. By the induction hypothesis, we can assume that $C_u$ (resp. $C_v$) has a 2-linear coloring $c_u$ (resp. $c_v$) using the colors $\{1,2\}$, such that if $C_u$ (resp. $C_v$) is an odd cycle, then the only monochromatic vertex in $c_u$ (resp. $c_v$) is $u$ (resp. $v$) and both edges incident on $u$ (resp. $v$) are colored 1 in $c_u$ (resp. $c_v$); otherwise, $c_u$ (resp. $c_v$) contains no monochromatic vertices. Moreover, we assume that if $u$ is a degree 1 vertex in $C_u$, then the only edge incident on $u$ is colored 1 in $c_u$. We construct a 2-linear coloring of $G$ as follows. First, color the edges of $C_u$ with the colors they have in $c_u$ and the edges of $C_v$ with the colors they have in $c_v$. Now, we color $uv$ with 2 to obtain the required 2-linear coloring of $G$.

Next, suppose that $H$ is connected. If $H$ is an odd cycle, then by the induction hypothesis, there is a 2-linear coloring $c$ of $H$ in which $v$ is the only monochromatic vertex. Now coloring $uv$ with the color in $\Missing(v)$ gives a 2-linear coloring of $G$ with no monochromatic vertices. On the other hand, if $H$ is not an odd cycle, then by the induction hypothesis, there exists a 2-linear coloring $c$ of $H$ using the colors $\{1,2\}$ containing no monochromatic vertices. If there exists $i\in\Missing(u)\cup\Missing(v)$, then we can color $uv$ with $i$ to obtain a 2-linear coloring of $H$ containing no monochromatic vertices. So we assume that $\Missing(u)=\Missing(v)=\emptyset$. If for some color $i\in\{1,2\}$, there is no path of color $i$ between $u$ and $v$ in the coloring $c$ of $H$, then we color $uv$ with $i$ to obtain the required 2-linear coloring of $G$. So we can assume that there is both a path of color 1 and a path of color 2 between $u$ and $v$ in the coloring $c$ of $H$. Let $w\in N_G(v)\setminus\{u\}$. Let us assume without loss of generality that $c(vw)=1$. Now changing the color of $vw$ to 2 and then coloring $uv$ with 1 gives a 2-linear coloring of $H$ with no monochromatic vertices.
\hfill\qed
\end{proof}

\section{Linear arboricity of 2-degenerate graphs}
In this section, we investigate the linear arboricity of $2$-degenerate graphs and identify some classes of 2-degenerate graphs that have linear arboricity equal to $\left\lceil\frac{\Delta}{2}\right\rceil$. Since $\left\lceil \frac{\Delta}{2}\right \rceil =  \left\lceil \frac{\Delta+1}{2}\right \rceil$ when $\Delta$ is odd, the linear arboricity of $2$-degenerate graphs having odd maximum degree is exactly $\left\lceil \frac{\Delta}{2}\right \rceil$. This gives rise to the following natural question: Is the linear arboricity of every $2$-degenerate graph with even maximum degree also equal to $\left\lceil \frac{\Delta}{2}\right \rceil$? Since cycles are $2$-degenerate graphs that need at least 2 colors in any linear coloring, we can hope to get an upper bound of $\left\lceil\frac{\Delta}{2}\right\rceil$ on the linear arboricity of 2-degenerate graphs only when $\Delta \geq 3$. We first show that the linear arboricity of $2$-degenerate graphs is $\left \lceil \frac{\Delta}{2}\right \rceil$ when $\Delta \geq 5$ in Section~\ref{sec:morethan4}. We conjecture that every 2-degenerate graph having maximum degree 4 has a 2-linear coloring. Note that if this conjecture is true, then it will follow that every 2-degenerate graph has linear arboricity equal to $\left\lceil\frac{\Delta}{2}\right\rceil$ as long as $\Delta\geq 3$. In Section~\ref{sec:atmost4}, we show that certain kinds of 2-degenerate graphs having maximum degree 4 have 2-linear colorings.
\subsection{Maximum degree greater than 4}\label{sec:morethan4}
\begin{theorem}\label{greater4}
Let $G$ be any $2$-degenerate graph with $\Delta(G)\leq 2k$, where $k\geq 3$ is an integer. Then $\chi'_l(G)\leq k$.
\end{theorem}	
\begin{proof}
We prove this by induction on the number of edges of $G$.
If $E(G)=\emptyset$, then the theorem is trivially true. So let us assume that $|E(G)|\geq 1$. Then we know that there exists a pivot edge $uv$ in $G$, where $v$ is a pivot and $d_G(u)\leq 2$.
 First suppose that $d_G(v)<2k$.
 Then let $H=G-uv$. We can assume by the induction hypothesis that there is a $k$-linear coloring of $H$.
 As $d_H(v)\leq 2k-2$, either $|\Missing(v)\cup\Once(v)|\geq 2$ or $|\Missing(v)|\geq 1$. Note that $d_H(u)\leq 1$. We can now color $uv$ with either a color in $\Missing(v)$ (if $\Missing(v)\neq \emptyset$) or a color in $\Once(v)\setminus\Colors(u)$ to get a $k$-linear coloring of $G$.
 
 So we shall assume that $d_G(v)=2k$. Observe that as $k\geq 3$, we have $d_G(v)\geq 6$, which means that there are at least 4 pivot edges incident to $v$ in $G$. Let $vw,vx$ be two pivot edges distinct from $uv$ and let $H=G-\{uv,vw,vx\}$. By the induction hypothesis, there exists a $k$-linear coloring $c$ of $H$. Since $d_H(v)=2k-3$, we have either $|\Once(v)|=3$ or $|\Missing(v)|=|\Once(v)|=1$. Let us first consider the case when $|\Once(v)|=3$. Let $\Once(v)=\{c_0,c_1,c_2\}$. Since for each $j\in\{0,1,2\}$, there can be a path colored $c_j$ from $v$ to at most one vertex in $\{v,w,x\}$, and $d_H(u)=d_H(w)=d_H(x)=1$, there exists a bijection $f:\{u,w,x\}\rightarrow\{0,1,2\}$ such that for each $z\in\{u,w,x\}$, if there is a path colored $c_j$ (where $j\in\{0,1,2\}$) having $v$ and $z$ as endvertices, then $f(z)=j$ (such a bijection $f$ can be constructed as follows: for each vertex $z\in\{u,w,x\}$ such that there is a path colored $c_j$ having endvertices $v$ and $z$, we assign $f(z)=j$, and we map the remaining vertices to arbitrarily chosen distinct integers from $\{0,1,2\}$ that have not been assigned so far). Now we color each edge $vz$, where $z\in\{u,w,x\}$, with the color $c_{(f(z)+1)\mod 3}$. This gives a valid $k$-linear coloring of $G$. Now consider the case when $|\Missing(v)|=|\Once(v)|=1$. Let $\Missing(v)=\{a\}$ and $\Once(v)=\{b\}$. Since $d_H(u)=d_H(w)=d_H(x)=1$, at most two vertices in $\{u,w,x\}$ can be present in any given monochromatic path. So we can assume without loss of generality that there is no monochromatic path colored $a$ having $u$ as an endvertex and one of $\{w,x\}$ as its other endvertex. We know that one of the vertices $\{w,x\}$ is not the other endvertex of the path colored $b$ starting at $v$. Let us assume without loss of generality that $w$ is such a vertex. Now we color $uv,vx$ with $a$ and $vw$ with $b$ to get a $k$-linear coloring of $G$.
 \hfill\qed
\end{proof}
\begin{corollary}\label{greaterequal5}
For any 2-degenerate graph $G$, $\chi'_l(G)=\left\lceil\frac{\Delta(G)}{2}\right\rceil$ when $\Delta(G)\geq 5$.
\end{corollary}

\subsection{Maximum degree at most 4}\label{sec:atmost4}
As noted before, the upper bound given by Corollary~\ref{greaterequal5} does not hold for maximum degree 2. We conjecture that every 2-degenerate graph $G$ having $\Delta(G)\leq 4$ has $\chi'_l(G)\leq 2$.

\begin{conj}\label{con:2-deg}
Every 2-degenerate graph $G$ with maximum degree at most $4$ has a 2-linear coloring.
\end{conj}
\medskip

In the following sections, we show that the conjecture is true for some subclasses of 2-degenerate graphs. In fact, in each of these cases, we exhibit the presence of a 2-linear coloring satisfying some additional properties.

\subsubsection{Graphs with a large number of edges}
~\medskip

We now prove that the conjecture is true if $|E(G)|\geq 2|V(G)|-5$.
\begin{definition}
We define $\mathcal{F}_k=\{G:$ where $G$ is a 2-degenerate graph of maximum degree at most 4 and $|E(G)|=2|V(G)|-k\}$.
\end{definition}
We observe that $\{\mathcal{F}_k:k\geq 2\}$ is a partition of the set of 2-degenerate graphs of maximum degree at most 4 (note that the only 2-degenerate graph having $2n-2$ edges is the single vertex graph).

\begin{observation}\label{class}
For all $G\in \mathcal{F}_k$ and any vertex $v\in V(G)$ having degree at most 2, we have $G-v \in \mathcal{F}_l$ where $l\leq k$. 
\end{observation}

\begin{lemma}\label{connected}
Let $G\in\bigcup_{j=2}^{5}\mathcal{F}_j$. Then $G$ can have at most 2 components of which at least one is a single vertex.
\end{lemma}
\begin{proof}
Let $G_1,G_2,\ldots,G_r$ be the subgraphs of $G$ induced by the connected components of $G$. By Observation~\ref{class}, we have that for each $i\in\{1,2,\ldots,r\}$, $G_i\in [l_i]$, where $2\leq l_i\leq 5$. It can be seen that $|E(G)|=2(|V(G_1)|+|V(G_2)|+\cdots+|V(G_r)|)-(l_1+l_2+\cdots+l_r)$. Clearly, if $r\geq 3$, then we have $|E(G)|\leq 2|V(G)|-6$, which is a contradiction to the fact that $G\in\bigcup_{j=2}^5 \mathcal{F}_j$. So we can assume that $r=2$. Now if both $l_1,l_2\geq 3$, then we again have $|E(G)|\leq 2|V(G)|-6$, leading to the same contradiction. Thus we conclude that at least one of $l_1,l_2$ is equal to 2, or in other words, one of the subgraphs $G_1,G_2$ is a single vertex.\hfill\qed
\end{proof}

We now show that Conjecture~\ref{con:2-deg} is true if $|E(G)|\geq 2|V(G)|-5$.

\begin{theorem}\label{lessequal4}
Every 2-degenerate graph $G$ with maximum degree at most 4 having at least $2|V(G)|-5$ edges has a 2-linear coloring in which there is at most one monochromatic vertex.
\end{theorem}
\begin{proof}
Let $2\leq k\leq 5$. We show that every graph $G\in \mathcal{F}_k$ has a 2-linear coloring in which $G$ contains at most one monochromatic vertex.
Our proof will be by induction on $|V(G)|$. We shall see that our induction step can be carried out to produce a graph with a smaller number of vertices than $G$ as long as $|V(G)|\geq 2$ and $G$ contains either a vertex of degree at most 1 or a pivot of degree at least 3. Note that by Observation~\ref{obs:pivot}, a pivot of degree at least 3 always exists in $G$ if $\Delta(G)\geq 3$. So the base case (when the induction step cannot be applied) is when every vertex of $G$ is of degree 2, or $G$ contains just a single vertex. In the latter case, the statement is trivially true. In the former case, $G$ is a disjoint union of cycles. By Lemma~\ref{connected}, we know that in this case, at most one component of $G$ can be an odd cycle, which means that there is a 2-linear coloring of $G$ with at most one monochromatic vertex (the edges of an even cycle can be given alternating colors so that no monochromatic vertex arises); so the statement is true in this case as well. This proves the base case. We shall now describe the induction step. We assume that the graph $G$ contains more than one vertex and that $G$ is not 2-regular.

Suppose that there is a vertex $u$ such that $d_G(u)\leq 1$. Let $H=G-u$. By Observation~\ref{class}, $H\in \mathcal{F}_l$, where $l\leq k$. Then by the induction hypothesis there is a 2-linear coloring of $H$ with at most one monochromatic vertex. If in $G$, there is an edge $uv$ incident on $u$, then color $uv$ with a color in $\Missing(v)$ if $d_H(v)\leq 1$, and with a color in $\Once(v)\cup\Missing(v)$ otherwise (such a color will exist as $d_H(v)\leq 3$). It is easy to see that this gives a 2-linear coloring of $G$ with at most one monochromatic vertex. So we shall assume that every vertex in $G$ has degree at least 2.

Suppose that $G$ contains a pivot $v$ such that $d_G(v)=3$. Let $uv$ be a pivot edge. Then $d_G(u)=2$ by the assumption above. Let $H=G-u$. It follows that $H\in \mathcal{F}_k$. Then using the induction hypothesis, there is a 2-linear coloring of $H$ using the colors $\{1,2\}$ and containing at most one monochromatic vertex. Let $x\in V(G)$ such that $N(u)=\{v,x\}$. We first color $ux$ with a color $p$ in the following way. If $d_H(x)\leq 1$, then we choose $p$ to be a color in $\Missing(x)$ and if $d_H(x)\geq 2$, then we choose $p$ to be a color in $p\in\Once(x)\cup\Missing(x)$ (note that this set is nonempty as $d_H(x)\leq 3$). We now have a 2-linear coloring of $H'=G-uv$ with at most one monochromatic vertex. Now we show how to extend this coloring to a 2-linear coloring of $G$ by assigning a color to the edge $uv$. Note that $d_{H'}(v)=2$. If $v$ is not a monochromatic vertex in $H'$, then we color $uv$ with a color in $\{1,2\}\setminus\{p\}$, so that we get a 2-linear coloring of $G$ with the same number of monochromatic vertices as there were in $H'$; so we are done. On the other hand, if $v$ is a monochromatic vertex in $H'$, then we color $uv$ with the color in $\Missing(v)$ so that we get a 2-linear coloring of $G$ which has at most one monochromatic vertex (note that even though $u$ may be a monochromatic vertex now, $v$ is no longer a monochromatic vertex). So we shall assume from here on that every pivot in $G$ has degree equal to 4.

\begin{figure}[h]
	\begin{center}
		\renewcommand{\vertexset}{(a,0,0),(w,-1,-1),(u,-1,1),(b,-2,0)}
		\renewcommand{\edgeset}{(u,a),(u,b),(a,w),(b,w)}
		\renewcommand{\defradius}{0.1}
		\renewcommand{\defthiedge}{0.5}
		\renewcommand{\defcoledge}{gray}
		\begin{tabular}{p{0.45\textwidth}p{0.45\textwidth}}
			\begin{center}
				\begin{tikzpicture}
				\draw [line width = 1] (0,0) -- (1,-1);
				\draw [line width = 1] (0,0) -- (1,1);
				\draw [line width = 1] (-2,0)-- (-3,1); 
				\draw [line width = 1] (-2,0)-- (-3,-1);
				\drawgraph
				\node [below=2] at (0,0) {$x$};
				\node [left=2] at (-2,0) {$v$};
			
				\node [above=2] at (-1,1) {$u$};
				\node [below=2] at (-1,-1) {$w$};
				\node [left] at (0.5,0.75) {1};
				\node [below] at (0.5,-0.75) {1};
				\node [right] at (-0.5,-0.90) {2};
				\node [right] at (-0.5,0.90) {2};
				\node [right] at (-1.7,0.90) {2};
				\node [right] at (-1.7,-0.90) {1};
				\node [ left=2] at (-2.5,.4) {1};
				\node [ left=2] at (-2.5,-0.4) {2};
				\end{tikzpicture}\\(a)
			\end{center}
			&
			\begin{center}
			\begin{tikzpicture}
			\draw [line width = 1] (0,0) -- (1,-1);
			\draw [line width = 1] (0,0) -- (1,1);
			\draw [line width = 1] (-2,0)-- (-3,1); 
			\draw [line width = 1] (-2,0)-- (-3,-1);
			\drawgraph
			\node [below=2] at (0,0) {$x$};
			\node [ left=2] at (-2,0) {$v$};
			
			\node [above=2] at (-1,1) {$u$};
			\node [below=2] at (-1,-1) {$w$};
			\node [left] at (0.5,0.75) {1};
			\node [below] at (0.5,-0.75) {2};
			\node [right] at (-0.5,-0.90) {2};
			\node [right] at (-0.5,0.90) {1};
			\node [right] at (-1.7,0.90) {2};
			\node [right] at (-1.7,-0.90) {1};
			\node [right] at (-2.7,-0.90) {1};
			\node [right] at (-2.7,0.90) {2};
				\end{tikzpicture}\\(b)
			\end{center}
		\end{tabular}
	\end{center}
	\caption{Constructing a 2-linear coloring when $N(u)=N(w)=\{v,x\}$.}
	\label{fig:C4}
\end{figure}
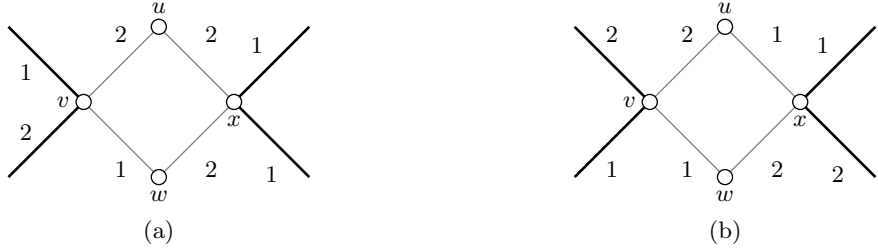
Let $v$ be a pivot in $G$. By our assumption, we have that $d_G(v)=4$. Since we have also assumed that every vertex in $G$ has degree at least 2, we know that there exist two neighbors $u,w$ of $v$ such that $d_G(u)=d_G(w)=2$. We first focus on the case when $N(u)=N(w)$. Suppose that $N(u)=N(w)=\{v,x\}$. Note that $x$ is also a pivot in $G$, which implies that $d_G(x)=4$; so $v$ and $x$ are symmetric. Consider the graph $H=G-\{u,w\}$. It can be seen that $H\in \mathcal{F}_k$. Using the induction hypothesis, we can conclude that there is a 2-linear coloring of $H$ using the colors $\{1,2\}$ and having at most one monochromatic vertex. In particular, at most one of $v,x$ can be a monochromatic vertex in $H$. Let us assume without loss of generality that $v$ is not a monochromatic vertex.
Now we color the edges $xu,xw,vu$ as shown in Figure~\ref{fig:C4}(a) if $x$ is monochromatic, and as shown in Figure~\ref{fig:C4}(b) if $x$ is not monochromatic.
It is easy to see that both cases, we obtain a 2-linear coloring of $G$ containing at most one monochromatic vertex.

\begin{figure}[h]
	\begin{center}
		\renewcommand{\vertexset}{(a,0,0),(x,-2,1),(u,-1,1.5),(v,-1,-1.5),(y,-2,-1),}
		\renewcommand{\edgeset}{(u,x),(u,a),(a,v),(v,y)}
		\renewcommand{\defradius}{0.1}
		\renewcommand{\defthiedge}{0.5}
		\renewcommand{\defcoledge}{gray}
		\begin{tabular}{p{0.33\textwidth}p{0.33\textwidth}p{0.33\textwidth}}
							\begin{center}
								\begin{tikzpicture}
								\draw [line width = 1] (-2,1) -- (-3,1);
								\draw [line width = 1] (-2,1) -- (-2,2);
								\draw [line width = 1] (-2,-1) -- (-3,-1);
								\draw [line width = 1] (-2,-2) -- (-2,-1);
								\draw [line width = 1] (0,0) -- (1,-1);
								\draw [line width = 1] (0,0) -- (1,1);
								\drawgraph
								\node [below=2] at (0,0) {$v$};
								\node [above left] at (-2,1) {$x_u$};
								\node [below left] at (-2,-1) {$x_w$};
								\node [above=2] at (-1,1.5) {$u$};
								\node [above=2] at (-1,-1.5) {$w$};
								\end{tikzpicture}\\$G$
							\end{center}
				&
					\begin{center}
						\renewcommand{\edgeset}{(u,x),(v,y)}
						\begin{tikzpicture}
						\draw [line width = 1] (-2,1) -- (-3,1);
						\draw [line width = 1] (-2,1) -- (-2,2);
						\draw [line width = 1] (-2,-1) -- (-3,-1);
						\draw [line width = 1] (-2,-2) -- (-2,-1);
						\draw [line width = 1] (0,0) -- (1,-1);
						\draw [line width = 1] (0,0) -- (1,1);
						\drawgraph
						\node [below=2] at (0,0) {$v$};
						\node [above left] at (-2,1) {$x_u$};
						\node [below left] at (-2,-1) {$x_w$};
						\node [above=2] at (-1,1.5) {$u$};
						\node [above=2] at (-1,-1.5) {$w$};
						\end{tikzpicture}\\$G'$
					\end{center}
		&
			\begin{center}
				\renewcommand{\vertexset}{(a,0,0),(x,-2,1),(u,-1,0),(y,-2,-1),}
				\renewcommand{\edgeset}{(u,x),(u,y)}
				\begin{tikzpicture}
				\draw [line width = 1] (-2,1) -- (-3,1);
				\draw [line width = 1] (-2,1) -- (-2,2);
				\draw [line width = 1] (-2,-1) -- (-3,-1);
				\draw [line width = 1] (-2,-2) -- (-2,-1);
				\draw [line width = 1] (0,0) -- (1,-1);
				\draw [line width = 1] (0,0) -- (1,1);
				\drawgraph
				\node [below=2] at (0,0) {$v$};
				\node [above left] at (-2,1) {$x_u$};
				\node [below left] at (-2,-1) {$x_w$};
				\node [below=2] at (-1,0) {$u$};
				\end{tikzpicture}\\$H$
			\end{center}
		\end{tabular}
	\end{center}
	\caption{Constructing the graphs $G'$ and $H$ from $G$.}
	\label{fig:C5}
\end{figure}
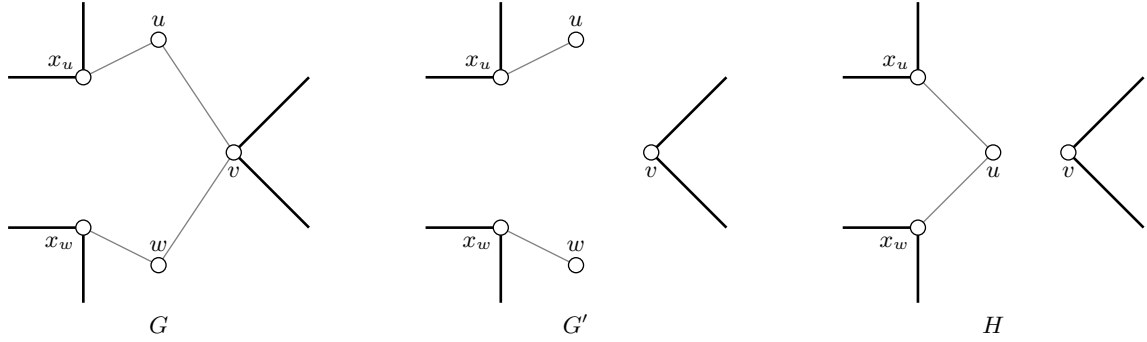
Next, we address the case when $N(u)\neq N(w)$. In this case, $N(u)\cap N(w)=\{v\}$. Let $N(u)\setminus N(w)=\{x_u\}$ and $N(w)\setminus N(u)=\{x_w\}$. Let $G'=G-\{vu,vw\}$. Now we identify the vertex $w$ with $u$ to get a new graph $H$. Formally, $H=G'/(u,w)$. Note that $H$ will not contain any parallel edges or self loops. The construction of $G'$ and $H$ from $G$ is shown in Figure~\ref{fig:C5}. It is also easy to see that $H\in \mathcal{F}_k$. Then by the induction hypothesis, there is a 2-linear coloring of $H$ using the colors $\{1,2\}$ and containing at most one monochromatic vertex. We now split back the identified vertex $u$ back to the vertices $u,w$ to get back the graph $G'$. We give the colors of the edges $ux_u,ux_w$ in $H$ to the edges $ux_u,wx_w$ respectively in $G'$. We now have a 2-linear coloring of $G'$ containing at most one monochromatic vertex. Now we show how to add the edges $vu$ and $vw$ to $G'$ and color them so that we get a 2-linear coloring of $G$ containing at most one monochromatic vertex. First let us suppose that $v$ is monochromatic in $G'$. Then it is monochromatic in $H$ as well, and therefore by the induction hypothesis, $u$ is not monochromatic in $H$. Thus the edges $ux_u$ and $wx_w$ have different colors in $G'$.
We color the edges $vu$ and $vw$ with the color in $\Missing(v)$ to obtain a 2-linear coloring of $G$ with exactly one monochromatic vertex (it will be either $u$ or $w$). Next, let us consider the case when $u$ is monochromatic in $H$. Then $v$ is not monochromatic in $H$ and therefore also in $G'$. So we have $\Colors(v)=\{1,2\}$ and we can assume without loss of generality that both the edges $ux_u$ and $wx_w$ are colored 1 in $G'$. Note that since $u$ and $w$ are distinct, there cannot be two monochromatic paths colored 1 in $G'$, one having endpoints $u,v$ and the other having endpoints $w,v$. If there is a monochromatic path colored 1 starting at $u$ and ending at $v$, then we color $uv$ with 2 and $vw$ with 1, otherwise we color $uv$ with 1 and $vw$ with 2. It is easy to see that we now have a 2-linear coloring of $G$ containing exactly one monochromatic vertex (it is either $u$ or $w$). Now we come to the last case, that is when neither $v$ nor $u$ are monochromatic in $H$. Then we have $\Colors(v)=\{1,2\}$ and we can assume without loss of generality that the edge $ux_u$ is colored 1 and the edge $wx_w$ is colored 2 in $G'$. Then we color $uv$ with 2 and $vw$ with 1. Since we have not introduced any new monochromatic vertex, this gives a 2-linear coloring of $G$ containing at most one monochromatic vertex.\hfill\qed
\end{proof}	
\medskip

We shall say that a 2-degenerate graph $G$ is \emph{maximal} if it contains $2|V(G)|-3$ edges. It is not difficult to see that maximal 2-degenerate graphs are exactly the graphs that can be constructed starting from a triangle and adding vertices of degree 2 iteratively.

\begin{corollary}
Every maximal 2-degenerate graph with maximum degree at most 4 contains a Hamiltonian path.
\end{corollary}
\begin{proof}
Let $G$ be such a graph, or in other words, $G\in\mathcal{F}_3$. By Theorem~\ref{lessequal4}, $G$ has a 2-linear coloring. Since the edges in each color class of this coloring form a linear forest, we know that there are at most $|V(G)|-1$ edges of each color. As $G$ has $2|V(G)|-3$ edges, we can conclude that one of the two color classes in this coloring contains exactly $|V(G)|-1$ edges. These edges then form a spanning linear forest with one connected component, which is nothing but a Hamiltonian path in $G$.\hfill\qed
\end{proof}

Both the conditions in the above corollary are necessary: the graph $K_{2,4}$ and the graph in Figure~\ref{fig:nonham} (both 2-degenerate graphs) have no Hamiltonian path. Note that the former has maximum degree 4 but is not maximal, whereas the latter is maximal but has maximum degree more than 4.

\begin{figure}
    \centering
    \renewcommand{\vertexset}{(a,0,2.5),(b,0,-2.5),(c,-2,0),(d,2,0),(e,4,0),(f,6,0)}
    \renewcommand{\edgeset}{(a,b),(a,c),(b,c),(a,f),(a,d),(a,e),(b,f),(b,e),(b,d)}
    \renewcommand{\defradius}{0.15}
    \begin{tikzpicture}[scale=0.5]
    \drawgraph
    \end{tikzpicture}
    \caption{A 2-degenerate graph of maximum degree 5 which has no Hamiltonian path.}\label{fig:nonham}
  
\end{figure}
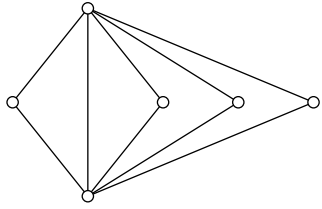
\subsubsection{Bipartite graphs}
~\medskip

In this section, we prove that Conjecture~\ref{con:2-deg} is true for bipartite graphs. In fact, we can show that such graphs have a 2-linear coloring containing no monochromatic vertex. The proof essentially follows the same arguments as in the proof of Theorem~\ref{lessequal4}.

\begin{theorem}\label{thm:bipartite}
Any bipartite 2-degenerate graph of maximum degree at most 4 has a 2-linear coloring without any monochromatic vertex.
\end{theorem}
\begin{proof}
Let $G$ be a bipartite 2-degenerate graph having $\Delta(G)\leq 4$.
Our proof will be by induction on $|V(G)|$.
The base case is when the graph contains just one vertex or is 2-regular, i.e. when it is a disjoint union of even cycles. It is easy to see that in both cases, the graph has a 2-linear coloring without any monochromatic vertices. Suppose first that $G$ contains a vertex $u$ such that $d_G(u)\leq 1$. By the induction hypothesis, there is a 2-linear coloring of the graph $H=G-u$ that does not contain any monochromatic vertex. We extend this to a 2-linear coloring of $G$ as follows. If $d_G(u)=0$, then there is nothing to do. If $d_G(u)=1$, then the edge $uv$ incident on $u$ is colored with the color in $\Missing(v)$ if $d_H(v)=1$, and is colored with a color in $\Once(v)\cup\Missing(v)$ if $d_H(v)\geq 2$. We now have a 2-linear coloring of $G$ containing no monochromatic vertex. So we shall assume from here on that every vertex has degree at least 2 in $G$. Next, suppose that $\Delta(G)\geq 3$. Then by Observation~\ref{obs:pivot} there exists a pivot $v$ in $G$ such that $d_G(v)\geq 3$. Suppose that $d_G(v)=3$.  Let $uv$ be a pivot edge. Then we let $H=G-uv$ and use the induction hypothesis to construct a 2-linear coloring of $H$ using colors $\{1,2\}$ and containing no monochromatic vertex. Note that $d_H(v)=2$ and $d_H(u)\leq 1$. Since there are no monochromatic vertices in the coloring, we have $\Colors(v)=\{1,2\}$. We color $uv$ with a color in $\{1,2\}\setminus\Colors(u)$ to get a 2-linear coloring of $G$ without any monochromatic vertex. Next, suppose that $d_G(v)=4$. Then we have at least two pivot edges $uv$ and $wv$ incident to $v$. Since we have assumed that every vertex has degree at least 2, we have $d_G(u)=d_G(w)=2$. Suppose that $N(u)=N(w)$. Then there exists a vertex $x$ such that $N(u)\cap N(w)=\{v,x\}$. Let $H=G-\{u,w\}$. By the induction hypothesis, there is a 2-linear coloring of $H$ using the colors $\{1,2\}$ and containing no monochromatic vertices. Clearly, $\Colors(v)=\Colors(x)=\{1,2\}$, as $H$ contains no monochromatic vertex. We color $xu,vw$ with 1 and $xw,vu$ with 2 to obtain the required 2-linear coloring of $G$. Finally, suppose that $N(u)\neq N(w)$. Let $N(u)\setminus N(w)=\{x_u\}$ and $N(w)\setminus N(u)=\{x_w\}$. Let $G'=G-\{vu,vw\}$ and let $H=G'/(u,w)$. By the induction hypothesis, there is a 2-linear coloring of $H$ using the colors $\{1,2\}$ and containing no monochromatic vertices. Consider the coloring of $G'$ obtained by retaining the color of $ux_u$ and giving the color of $ux_w$ to $wx_w$. Since $u$ was not a monochromatic vertex in $H$, we can assume without loss of generality that in the coloring of $G'$, the edge $ux_u$ is colored 1 and the edge $wx_w$ is colored 2. As $v$ was not a monochromatic vertex in $H$, it is not monochromatic in $G'$ either. So, we can color $uv$ with $2$ and $wv$ with 1 to obtain the required 2-linear coloring of $G$.
\hfill\qed
\end{proof}

\subsubsection{Partial 2-trees}
~\medskip

A graph $G$ is said to be a partial 2-tree if it has a tree decomposition of width at most 2. It is well known that partial 2-trees are 2-degenerate and are closed under taking minors. The following folklore result can be proved using this fact.
\begin{observation}\label{obs:configs}
Every partial 2-tree $G$ contains one of the following configurations:
\begin{enumerate}
\renewcommand{\labelenumi}{(\alph{enumi})}
\renewcommand{\theenumi}{(\alph{enumi})}
\item\label{degreeone} A vertex of degree at most one,
\item\label{shortcircuit} Two adjacent vertices each of degree two,
\item\label{C4} Two non-adjacent vertices of degree two that have the same set of neighbors,
\item\label{triangle} A triangle containing a vertex of degree 2, a vertex of degree 3, and a vertex of degree at least 3,
\item\label{2triangle} Two triangles having no common edges, but having a common vertex with degree 4, and in each triangle, apart from the common vertex, there is a vertex of degree 2 and a vertex of degree at least 4.
\end{enumerate}
\end{observation}

\begin{definition}
Let $G$ be a graph and let $V'\subseteq V(G)$ be the set of vertices of degree 2 in $G$. Given a set $\mathcal{S}=\{S_1,S_2,\ldots,S_k\}$ of pairwise disjoint subsets of $V'$, each of size 2, a 2-linear coloring of $G$ is said to \emph{satisfy} $\mathcal{S}$ if at most one vertex in each $S_i$, where $1\leq i\leq k$, is monochromatic.
\end{definition}

We refer to a collection of sets of the form $\mathcal{S}$ defined above, as a ``disjoint collection of pairs of degree 2 vertices in $G$''. 

\begin{theorem}\label{thm:p2tree}
Let $G$ be a partial 2-tree having maximum degree at most 4.
If $\mathcal{S}$ is a disjoint collection of pairs of degree 2 vertices of $G$,
then $G$ has a 2-linear coloring that satisfies $\mathcal{S}$.
\end{theorem}
\begin{proof}
We shall prove the statement of the theorem by induction on $|V(G)|$. The base case when $G$ contains a single vertex is trivially true. By Observation~\ref{obs:configs}, $G$ contains one of the configurations \ref{degreeone}--\ref{2triangle} described in the statement of the observation. In each case, we shall show how to compute a 2-linear coloring of $G$ satisfying $\mathcal{S}$.

\begin{figure}[h]
\begin{center}
\renewcommand{\vertexset}{(a,0,0),(x,-2,0),(u,-1,1.5),(y,2,0),(v,1,1.5)}
\renewcommand{\edgeset}{(a,x),(u,x),(u,a),(a,y),(a,v),(y,v)}
\renewcommand{\defradius}{0.1}
\renewcommand{\defthiedge}{0.5}
\renewcommand{\defcoledge}{gray}
\begin{tabular}{p{0.45\textwidth}p{0.45\textwidth}}
\begin{center}
\begin{tikzpicture}
\draw [line width = 1] (-2,0) -- (-3,0);
\draw [line width = 1] (-2,0) -- (-2,-1);
\draw [line width = 1] (2,0) -- (3,0);
\draw [line width = 1] (2,0) -- (2,-1);
\drawgraph
\node [below=2] at (0,0) {$a$};
\node [above left] at (-2,0) {$x$};
\node [above right] at (2,0) {$y$};
\node [above=2] at (-1,1.5) {$u$};
\node [above=2] at (1,1.5) {$v$};
\node [left] at (-2,-0.75) {1};
\node [below] at (-2.75,0) {1};
\node [right] at (2,-0.75) {1};
\node [below] at (2.75,0) {1};
\node [below] at (-1,0) {2};
\node [below] at (1,0) {2};
\node at (-1.7,0.9) {2};
\node at (1.7,0.9) {2};
\node at (-0.3,0.9) {1};
\node at (0.3,0.9) {1};
\end{tikzpicture}\\(i)
\end{center}
&
\begin{center}
\begin{tikzpicture}
\draw [line width = 1] (-2,0) -- (-3,0);
\draw [line width = 1] (-2,0) -- (-2,-1);
\draw [line width = 1] (2,0) -- (3,0);
\draw [line width = 1] (2,0) -- (2,-1);
\drawgraph
\node [below=2] at (0,0) {$a$};
\node [above left] at (-2,0) {$x$};
\node [above right] at (2,0) {$y$};
\node [above=2] at (-1,1.5) {$u$};
\node [above=2] at (1,1.5) {$v$};
\node [below=2] at (0,0) {$a$};
\node [above left] at (-2,0) {$x$};
\node [above right] at (2,0) {$y$};
\node [above=2] at (-1,1.5) {$u$};
\node [above=2] at (1,1.5) {$v$};
\node [left] at (-2,-0.75) {1};
\node [below] at (-2.75,0) {1};
\node [right] at (2,-0.75) {2};
\node [below] at (2.75,0) {2};
\node [below] at (-1,0) {2};
\node [below] at (1,0) {1};
\node at (-1.7,0.9) {2};
\node at (1.7,0.9) {1};
\node at (-0.3,0.9) {1};
\node at (0.3,0.9) {2};
\end{tikzpicture}\\(ii)
\end{center}
\\
\begin{center}
\begin{tikzpicture}
\draw [line width = 1] (-2,0) -- (-3,0);
\draw [line width = 1] (-2,0) -- (-2,-1);
\draw [line width = 1] (2,0) -- (3,0);
\draw [line width = 1] (2,0) -- (2,-1);
\drawgraph
\node [below=2] at (0,0) {$a$};
\node [above left] at (-2,0) {$x$};
\node [above right] at (2,0) {$y$};
\node [above=2] at (-1,1.5) {$u$};
\node [above=2] at (1,1.5) {$v$};
\node [below=2] at (0,0) {$a$};
\node [above left] at (-2,0) {$x$};
\node [above right] at (2,0) {$y$};
\node [above=2] at (-1,1.5) {$u$};
\node [above=2] at (1,1.5) {$v$};
\node [left] at (-2,-0.75) {1};
\node [below] at (-2.75,0) {2};
\node [right] at (2,-0.75) {1};
\node [below] at (2.75,0) {2};
\node [below] at (-1,0) {1};
\node [below] at (1,0) {2};
\node at (-1.7,0.9) {2};
\node at (1.7,0.9) {1};
\node at (-0.3,0.9) {1};
\node at (0.3,0.9) {2};
\end{tikzpicture}\\(iii)
\end{center}
&
\begin{center}
\begin{tikzpicture}
\draw [line width = 1] (-2,0) -- (-3,0);
\draw [line width = 1] (-2,0) -- (-2,-1);
\draw [line width = 1] (2,0) -- (3,0);
\draw [line width = 1] (2,0) -- (2,-1);
\drawgraph
\node [below=2] at (0,0) {$a$};
\node [above left] at (-2,0) {$x$};
\node [above right] at (2,0) {$y$};
\node [above=2] at (-1,1.5) {$u$};
\node [above=2] at (1,1.5) {$v$};
\node [below=2] at (0,0) {$a$};
\node [above left] at (-2,0) {$x$};
\node [above right] at (2,0) {$y$};
\node [above=2] at (-1,1.5) {$u$};
\node [above=2] at (1,1.5) {$v$};
\node [left] at (-2,-0.75) {1};
\node [below] at (-2.75,0) {2};
\node [right] at (2,-0.75) {1};
\node [below] at (2.75,0) {1};
\node [below] at (-1,0) {2};
\node [below] at (1,0) {2};
\node at (-1.7,0.9) {1};
\node at (1.7,0.9) {2};
\node at (-0.3,0.9) {1};
\node at (0.3,0.9) {1};
\end{tikzpicture}\\(iv)
\end{center}
\end{tabular}
\end{center}
\caption{Constructing a 2-linear coloring of $G$ from a 2-linear coloring of $H$ in case~\ref{2triangle}.}
\label{fig:p2-tree}
\end{figure}
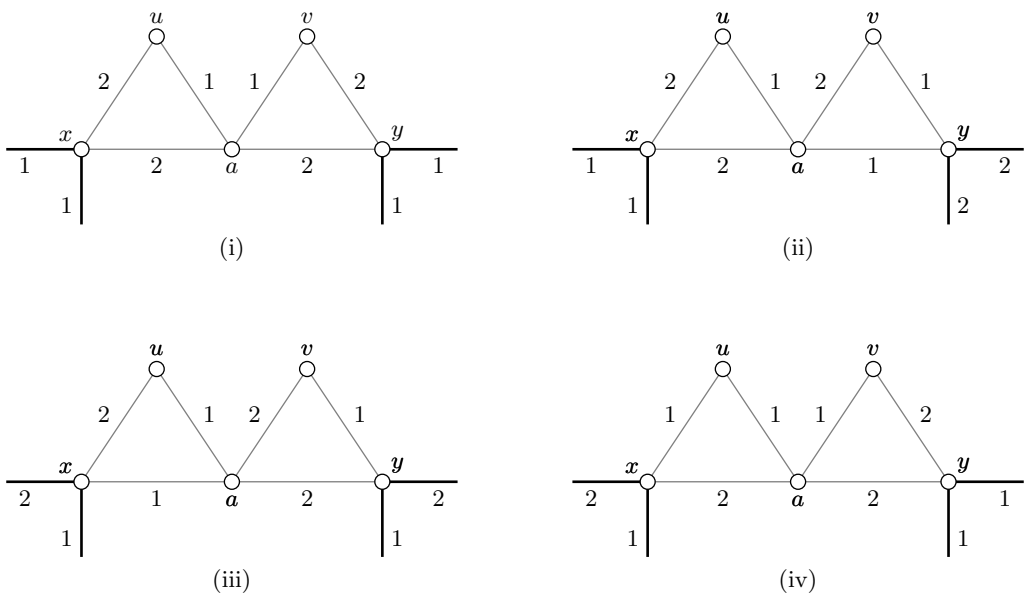

\ref{degreeone} Suppose that $G$ contains a vertex $u$ of degree at most one, and let $N(u)=\{v\}$. Let $c$ be a 2-linear coloring of the partial 2-tree $H=G-u$ using the colors $\{1,2\}$ satisfying $\mathcal{S}\setminus\{S\colon v\in S\}$, which exists by the induction hypothesis. We now construct the required 2-linear coloring of $G$ as follows. Assign every edge $e\in E(G)\setminus\{uv\}$ the color $c(e)$, and color the edge $uv$ with the color in $\Missing(v)$ if $d_G(v)=2$ and with a color in $\{1,2\}\setminus\Twice(v)$ otherwise (such a color will exist as $d_H(v)\leq 3$). It is easy to see that this is a 2-linear coloring. Note that if $d_G(v)=2$, then $v$ is not a monochromatic vertex in the coloring that is constructed. Therefore, the 2-linear coloring constructed satisfies $\mathcal{S}$.

\ref{shortcircuit} Suppose that $G$ contains two adjacent vertices $u,v$ such that $d(u)=d(v)=2$. Consider the partial 2-tree $H=G-uv$. If there exist $u',v'\in V(G)$ such that $\{u,u'\},\{v,v'\}\in\mathcal{S}$, let $\mathcal{S}'=(\mathcal{S}\cup\{\{u',v'\}\})\setminus\{\{u,u'\},\{v,v'\}\}$, otherwise let $\mathcal{S}'=\mathcal{S}\setminus\{S\colon u\in S$ or $v\in S\}$. By the induction hypothesis, there exists a 2-linear coloring $c$ of $H$ that satisfies $\mathcal{S}'$. We extend this to a 2-linear coloring of $G$ by assigning a color to the edge $uv$. If $u'$ does not exist, $u'=v$, or if $u'$ is not monochromatic, then color $uv$ with a color so that $v$ is not monochromatic (since we can afford to let $u$ be monochromatic). Otherwise, color $uv$ with a color so that $u$ is not monochromatic. It is easy to see that we have a 2-linear coloring of $G$. Clearly, if one of $u',v'$ does not exist, or if $\{u',v'\}=\{u,v\}$, then the 2-linear coloring that we have constructed satisfies $\mathcal{S}$. Otherwise, $\{u',v'\}\in\mathcal{S}'$, and so only at most one of $u',v'$ can be monochromatic in the coloring $c$ of $H$ given by the induction hypothesis. So in this case too, we have a 2-linear coloring of $G$ that satisfies $\mathcal{S}$.

\ref{C4} Suppose that $G$ contains two non-adjacent degree 2 vertices $u,v$ such that $N(u)=N(v)=\{x,y\}$. Let $H=G-\{u,v\}$. Clearly, $H$ is a partial 2-tree. If there exist $u',v'\in V(G)$ such that $\{u,u'\},\{v,v'\}\in\mathcal{S}$, let $\mathcal{S}'=(\mathcal{S}\cup\{\{u',v'\}\})\setminus\{\{u,u'\},\{v,v'\}\}$, otherwise let $\mathcal{S}'=\mathcal{S}\setminus\{S\colon u\in S$ or $v\in S\}$. If $d_H(x)=d_H(y)=2$, then define $\mathcal{S}''=\mathcal{S}'\cup\{\{x,y\}\}$, otherwise define $\mathcal{S}''=\mathcal{S}'$. By the induction hypothesis, there exists a 2-linear coloring $c$ of $H$ that satisfies $\mathcal{S}''$. We construct a 2-linear coloring of $G$ as follows. Assign every edge $e\in E(G-\{u,v\})$ the color $c(e)$. Note that as $\{x,y\}\in\mathcal{S}''$, at least one of $x,y$ is non-monochromatic in the coloring $c$ of $H$. This means that in order to construct a 2-linear coloring of $G$, we can color the edges $ux,vx,uy,vy$ in such a way that $u$ is not monochromatic and also in such a way that $v$ is not monochromatic. As before, if $u'$ does not exist, $u'=v$ or if $u'$ is not monochromatic, then we color the edges $ux,vx,uy,vy$ such that $v$ is not monochromatic, otherwise we color those edges so that $u$ is not monochromatic. It is easy to verify that the 2-linear coloring constructed in this manner satisfies $\mathcal{S}$.

\ref{triangle} Suppose that the vertices $u,v,w$ form a triangle in $G$ and $d_G(u)=2$ and $d_G(v)=3$. Consider the partial 2-tree $H=G-\{u\}$. If there is $u'\in V(G)$ such that $\{u,u'\}\in \mathcal{S}$ then we let $\mathcal{S}'=(\mathcal{S}\setminus\{u,u'\})\cup \{\{v,u'\}\}$ otherwise we let $\mathcal{S}'=\mathcal{S}$. By the induction hypothesis there is a 2-linear coloring $c$ of $H$ using the colors $\{1,2\}$ that satisfies $\mathcal{S}'$. We will now construct a 2-linear coloring of $G$ by assigning colors to the edges $uv$ and $uw$. First we color $uw$ with an arbitrary color from $\Missing(w)\cup\Once(w)$ (this set is nonempty as $d_H(w)\leq 3$). Note that $d_H(v)=2$ and if $u'$ exists, only one of $u'$ or $v$ can be monochromatic in the coloring $c$ of $H$. If $v$ is monochromatic in $H$, then we color $uv$ with the color in $\Missing(v)$, and otherwise we color $uv$ with a color in $\{1,2\}$ that is different from the color of $uw$. It is easy to verify that we now have a 2-linear coloring of $G$ that satisfies $\mathcal{S}$. 

\ref{2triangle} 
Suppose that both $\{a,u,x\}$ and $\{a,v,y\}$ induce triangles in $G$, where $d_G(u)=d_G(v)=2$. Then we have $d_G(x)=d_G(y)=4$. We consider the partial 2-tree $G=H-\{u,v,a\}$. Note that $d_H(x)=d_H(y)=2$. If there exists $u'\neq v$ such that $\{u,u'\}\in\mathcal{S}$, then define $\mathcal{A}=(\mathcal{S}\setminus\{\{u,u'\}\})\cup\{\{y,u'\}\}$, otherwise let $\mathcal{A}=\mathcal{S}$. Now if there exists $v'\neq u$ such that $\{v,v'\}\in\mathcal{A}$, then we let $\mathcal{A}'=(\mathcal{A}\setminus\{\{v,v'\}\})\cup\{\{x,v'\}\}$. Finally, we define $\mathcal{S}'=\mathcal{A}'\setminus\{\{u,v\}\}$. Let $c$ be a 2-linear coloring of $H$ using the colors $\{1,2\}$ that satisfies $\mathcal{S}'$ that exists by the induction hypothesis. We shall assign colors to the edges $ux,ua,va,vy,ax,ay$ to construct a 2-linear coloring of $G$. If both $x$ and $y$ are monochromatic or if both $x$ and $y$ are non-monochromatic, then we can color these edges in such a way that both $u$ and $v$ are non-monochromatic as shown in Figures~\ref{fig:p2-tree}(i)--\ref{fig:p2-tree}(iii). If $y$ is monochromatic and $x$ is non-monochromatic, then we can color these edges as shown in Figure~\ref{fig:p2-tree}(iv) so that $v$ is non-monochromatic and $u$ is monochromatic. Symmetrically, if $x$ is monochromatic and $y$ is non-monochromatic, then we can color these edges in such a way that $u$ is non-monochromatic and $v$ is monochromatic. It is easy to verify that in all cases, we have a 2-linear coloring of $G$ that satisfies $\mathcal{S}$.\hfill\qed
\end{proof}	

\newcommand{\Edges}{\mathrm{\texttt{Edges}}}
\newcommand{\Adj}{\mathrm{\texttt{Adj}}}

\newcommand{\once}{\mathrm{\texttt{\textup{Onc}}}}
\newcommand{\missing}{\mathrm{\texttt{\textup{Miss}}}}
\newcommand{\missingindex}{\mathrm{\texttt{Ptrs}}}

\section{Linear time algorithms}\label{sec:linalgo}
We now describe how the proofs of all the upper bounds we have derived can be converted into linear-time algorithms that produce a linear coloring of an input 3-degenerate or 2-degenerate graph using at most the number of colors given by the corresponding upper bound. We take the algorithmic framework described in detail in~\cite{conf}.

\subsection{The 3-degenerate cases}
The proof of Theorem~\ref{thm:3-degrev} can be converted into a linear-time algorithm that produces a $\left\lceil\frac{\Delta(G)}{2}\right\rceil$-linear coloring of an input 3-degenerate graph $G$ having $\Delta(G)\geq 9$ as follows. The input to the algorithm is a 3-degenerate graph having maximum degree at most $k$, where $k\geq 9$, and the algorithm generates a $\left\lceil\frac{k}{2}\right\rceil$-coloring of the graph (which can be retrieved from the $\Edges$ list, as in the algorithm of~\cite{conf}) that does not contain any monochromatic vertices. At every step, the algorithm modifies the current graph to obtain a smaller graph on which it recurses to generate a $\left\lceil\frac{k}{2}\right\rceil$-coloring that does not contain monochromatic vertices, which is then extended to a $\left\lceil\frac{k}{2}\right\rceil$-coloring of the current graph that also does not contain any monochromatic vertices. Following the proof of Theorem~\ref{thm:3-degrev}, the algorithm converts the graph $G$ at any particular stage into the smaller graph $H'$ by picking a pivot $v$, removing all the pivot edges incident on $v$ as well as any edges between two neighbours of $v$ each having degree two, and then finally pairing up and identifying as many degree one vertices as possible that were earlier neighbours of $v$---in the language of the proof of Theorem~\ref{thm:3-degrev}, we remove all edges in $F\cup I$ and then pair up the vertices in $W$ and identify them to construct $H'$. It is not difficult to see that if the number of pivot edges incident on $v$ in $G$ is $t=|F|$, then $G$ can be modified into the graph $H'$ in $O(t)$ time. In fact, it is straightforward to see that the graph $G$ can be modified into the graph $H=G-(F\cup I)$ in $O(t)$ time. For converting $H$ to $H'$, we initialize a list $L$ of ``sets'' (which are again implemented as lists). Initially, $L=\{\{u\}\colon u\in W\}$. We now arbitrarily choose two vertices $u,u'$ from two different sets in $L$ and check if their single neighbour in $H$ is the same vertex. If yes, we merge the two sets from which $u$ and $u'$ were chosen. Otherwise, we identify $u$ with $u'$ and remove both $u$ and $u'$ from their respective sets in $L$. If some set in $L$ becomes empty in the process, we remove it from $L$. In $O(t)$ time, the list $L$ will either become empty or will contain only a single set. If $L$ is nonempty, then the single set that it contains is the set $W'$ from the proof of Theorem~\ref{thm:3-degrev}. The algorithm now recurses on $H'$ to produce a $\left\lceil\frac{k}{2}\right\rceil$-coloring of $H'$ that contains no monochromatic vertices. Once this is done, we split back the vertices that were identified to obtain the graph $H$ back from $H'$, but retaining the colors on the edges. Note that while doing this, no monochromatic path will get split, since the vertices in $H'$ that get split have degree two and are not monochromatic in the coloring of $H'$. This means that even though we identify vertices and split them back, we can work with just $\left\lceil\frac{k}{2}\right\rceil$-linear colorings and do not need ``pseudo $\left\lceil\frac{k}{2}\right\rceil$-linear colorings'' as in~\cite{conf}. Now it only remains to be shown that the edges in $F\cup I$ can be added back to $H$ and colored in just $O(t)$ time in order to obtain a $\left\lceil\frac{k}{2}\right\rceil$-linear coloring of $G$. Since $d_G(v)\leq t+3$, and coloring the edges of $I$ once we have the required $\left\lceil\frac{k}{2}\right\rceil$-linear coloring of $G-I$ is easy, we only need to show that the $\left\lceil\frac{k}{2}\right\rceil$-linear coloring of $H$ can be extended to a $\left\lceil\frac{k}{2}\right\rceil$-linear coloring of $G-I$ containing no monochromatic vertices in $O(d_G(v))$ time. We start by iterating through the available colors $1,2,\ldots,\left\lceil\frac{k}{2}\right\rceil$ in ascending order. For each color $i\in\{1,2,\ldots,\left\lceil\frac{k}{2}\right\rceil\}$, we assign $i$ to two as yet uncolored edges in $F$ if $i$ does not appear on the (at most three) edges incident on $v$ in $H$, and we assign $i$ to one as yet uncolored edge in $F$ if $i$ appears on exactly one of the edges incident on $v$ in $H$ (if $i$ occurs twice on the edges incident on $v$, we do not assign $i$ to any edges in $F$). Each time an edge $uv$ in $F$ gets colored with a color $i$ in this way, we update the color of that edge in the $\Edges$ list, and also modify the lists $\missing(v)$, $\once(v)$, $\missing(u)$, and $\once(u)$ as usual. But we do not modify any of the ``path objects''---i.e., we do not alter our record of the monochromatic paths in the coloring at this stage; in fact, the edge coloring of $G-I$ that we have at this stage may contain monochromatic cycles and/or monochromatic vertices, both of which we want to avoid. Nevertheless, it is clear that it takes only $O(d_G(v))$ time to color all the edges of $F$ in this way. We now follow the proof of Theorem~\ref{thm:3-degrev} in order to permute the colors on the edges in $F$ so as to get the required $\left\lceil\frac{k}{2}\right\rceil$-linear coloring of $G-I$. Note that the path objects can be used for detecting the presence of monochromatic cycles and also for eliminating them in the way described in the proof of Theorem~\ref{thm:3-degrev}.
It is not difficult to see that the same kind of data structures can be used to encode the partial linear colorings generated during the various stages of the algorithm implying that the linear arboricity of 3-degenerate graphs having maximum degree at least 9 is linear time computable. 

The same ideas can be adopted for converting the proofs of Theorem~\ref{thm:3dg7} and Theorem~\ref{thm:3dg5} into linear-time algorithms that generate linear colorings for 3-degenerate graphs of maximum degree 7 and 5 using at most 4 and at most 3 colors respectively. It is also straightforward to convert the proof of Theorem~\ref{thm:3dg3} into a linear-time algorithm that constructs a 2-linear coloring for 3-degenerate graphs of maximum degree 3.

\subsection{The 2-degenerate cases}
The proof of Theorem~\ref{greater4} can be converted into a linear time algorithm that outputs a $\left\lceil\frac{\Delta(G)}{2}\right\rceil$-linear coloring of an input 2-degenerate graph $G$ having $\Delta(G)\geq 5$ by utilizing the framework described in~\cite{conf} (note that the algorithm can just use $k$-linear colorings instead of the ``pseudo-$k$-linear colorings'' of~\cite{conf} as we never identify vertices and split them later, so no monochromatic paths ever get split into two paths). Note that this means that the problem of computing linear arboricity is linear time solvable in 2-degenerate graphs having maximum degree at least 5. It seems natural to ask the following question: is the problem of computing linear arboricity NP-hard when restricted to 2-degenerate graphs of maximum degree at most 4? We do not know the answer to this question, but would like to remark here that this problem seems similar to Conjecture~4 in~\cite{Cygan}.

It is not difficult to see that using the techniques and data structures described in~\cite{conf}, the proof of Theorem~\ref{lessequal4} can be converted into a linear time algorithm that given an input 2-degenerate graph $G$ having $\Delta(G)\leq 4$ and $|E(G)|\geq 2|V(G)|-5$ constructs a 2-linear coloring of $G$ containing at most one monochromatic vertex. Note that in this case, we do sometimes split back an identified vertex of degree 2 which may even be a monochromatic vertex. This means that we will need the full machinery of the algorithm of~\cite{conf}, including the use of pseudo 2-linear colorings, to implement the algorithm to run in linear time.

For converting the proof of Theorem~\ref{thm:bipartite} into a linear time algorithm that computes a 2-linear coloring with no monochromatic vertex for any input bipartite 2-degenerate graph having maximum degree at most 4, we do not need to pseudo 2-linear colorings, since even though we split back identified vertices, they are always vertices of degree 2 that are not monochromatic. So no monochromatic path gets split in the process, which means that we do not need to work with pseudo 2-linear colorings.

The proof of Theorem~\ref{thm:p2tree} can also be converted into a linear time algorithm that computes a 2-linear coloring of any input partial 2-tree having maximum degree at most 4 as explained below. Note that since we do not identify a pair of vertices and then split them back during the induction step, no monochromatic paths need to be split during the algorithm, which means that it can work with just 2-linear colorings instead of pseudo-2-linear colorings. For the algorithm to work in linear time, one of the five configurations listed in Observation~\ref{obs:configs} has to be found in $O(1)$ time during the induction step. This can be accomplished, for example, as follows. For any vertex $u$, we can determine whether it is part of a configuration of one of the types mentioned in Observation~\ref{obs:configs} by doing a local search starting from $u$. Since every vertex has at most 4 neighbors, this can be done in $O(1)$ time. At every stage of the algorithm, we maintain a list of all the configurations in the current graph. At the start, we compute the list of all configurations in the original graph $G$ in linear time by performing the local search starting from all vertices of $G$. With each vertex, we also store a list of pointers to the configurations that it is a part of. (Note that a vertex cannot be part of more than $O(1)$ configurations. Further, note that a naive approach may lead to the same configuration being detected more than once---the local search from each of its vertices may detect the configuration. Therefore the list of all configurations may contain duplicates, but this does not affect the runtime complexity of the algorithm as a configuration is duplicated only as many times as the number of vertices in it; i.e. $O(1)$ times.) During the induction step, we remove vertices or edges to obtain a smaller graph, and in this process the degrees of at most five vertices in the resultant graph will be different from their degrees in the original graph. We remove all configurations that these vertices were a part of, and then determine all configurations that they are a part of in the new graph by again doing a local search starting from them. This takes only $O(1)$ time. The only remaining detail is regarding the implementation of the list $\mathcal{S}$ of pairs of degree two vertices. We do not explicitly store the list $\mathcal{S}$, but instead, each pair $\{u,v\}$ of this list can be encoded by storing a pointer to $v$ on $u$ and a pointer to $u$ on $v$. It is easy to see that these pointers can be easily manipulated for achieving the modifications to the list $\mathcal{S}$ that we need.

\bigskip

\noindent\textbf{Acknowledgements:}
The first author was partially supported by the fixed grant scheme SERB-MATRICS project number MTR/2019/000790.

\bibliography{lac}
\end{document}